\documentclass[reqno,11pt]{amsart}
\usepackage{amssymb, amsmath,latexsym,amsfonts,amsbsy, amsthm}
\usepackage{color}
\usepackage{mathrsfs}
\usepackage{esint}
\usepackage{graphics,color}
\usepackage{enumerate}
\usepackage{marginnote}

\newcommand{\beq}{\begin{equation}}
\newcommand{\eeq}{\end{equation}}
\newcommand{\ben}{\begin{eqnarray}}
\newcommand{\een}{\end{eqnarray}}
\newcommand{\beno}{\begin{eqnarray*}}
\newcommand{\eeno}{\end{eqnarray*}}

\renewcommand{\theequation}{\thesection.\arabic{equation}}

\newtheorem{theorem}{Theorem}[section]

\newtheorem{lemma}[theorem]{Lemma}
\newtheorem{proposition}[theorem]{Proposition}

\newtheorem{Theorem}{Theorem}[section]
\newtheorem{Definition}[Theorem]{Definition}
\newtheorem{Proposition}[Theorem]{Proposition}
\newtheorem{Lemma}[Theorem]{Lemma}
\newtheorem{Corollary}[Theorem]{Corollary}
\newtheorem{Remark}[Theorem]{Remark}

\begin{document}

\title[Boussinesq Equation]
{H\"{o}lder continuous weak solution of Boussinesq equation with diffusive temperature}

\author{Tianwen Luo}
\address{School of  Mathematical Sciences, Yau Mathematical Science center, Tsinghua University, Beijing 100084, China}
\email{twluo@mail.tsinghua.edu.cn}

\author{Tao Tao}
\address{School of  Mathematics Sciences, Shandong University, jinan, 250100, China}
\email{taotao@amss.ac.cn}
\author{Liqun Zhang}
\address{Academy of Mathematic and System Science , CAS Beijing 100190, China; and School of Mathematical Sciences, University of Chinese Academy of Sciences, Beijing 100049, China}
\email{lqzhang@math.ac.cn}

%

\date{\today}
\maketitle

\renewcommand{\theequation}{\thesection.\arabic{equation}}
\setcounter{equation}{0}


\setcounter{equation}{0}

\begin{abstract}
We show the existence of H\"{o}lder continuous periodic weak solutions of the 2d Boussinesq equation with diffusive temperature which satisfy the prescribed kinetic energy. More precisely, for any smooth $e(t):[0,1]\rightarrow R_+$ and $\varepsilon\in (0, \frac{1}{10})$, there exist $v\in C^{\frac{1}{10}-\varepsilon}([0,1]\times {\rm T}^2), \theta\in C_t^{1,\frac{1}{20}-\frac{\varepsilon}{2}}C_x^{2,\frac{1}{10}-\varepsilon}([0,1]\times {\rm T}^2)$ which solve (\ref{e:boussinesq equation}) in the sense of distribution and satisfy
\begin{align}
e(t)=\int_{{\rm T}^2}|v(t,x)|^2dx, \quad \forall t\in [0,1].\nonumber
\end{align}
\end{abstract}


\vskip 0.2cm


\section{Introduction}
In this paper, we consider the following 2d Boussinesq equation
\begin{equation}\label{e:boussinesq equation}
\begin{cases}
\partial_tv+v\cdot\nabla v+\nabla p=\theta e_2,\quad\quad \mbox{in}\quad {\rm T}^2\times [0,1]\\
\hbox{div}v=0,\\
\partial_t\theta+v\cdot \nabla\theta-\triangle \theta=0, \quad\quad \mbox{in}\quad {\rm T}^2\times [0,1],
\end{cases}
\end{equation}
where ${\rm T}^2$ denotes the 2-dimensional torus, and $e_2=\binom{0}{1}$. Here in our notations, $v$ is the velocity vector, $p$ is the pressure, and $\theta$ denotes the temperature or density which is a scalar function.
The Boussinesq equation was introduced to model many geophysical flows, such as atmospheric fronts and ocean circulations (see, for example, \cite{Ma},\cite{Pe}).

The global well-posedness of strong solution has been established by many authors for the Cauchy problem of (\ref{e:boussinesq equation}) in 2d with regularity data \textcolor{blue}{(}see, for example, \cite{Chae}, \cite{Hou}). For the 3-dimensional case, the global existence of  smooth solution of (\ref{e:boussinesq equation}) remains open.

Moreover, the study of weak solutions in fluid dynamics, including those which fail to conserve energy, is quite natural in the context of turbulent flow, and has been conducted by many people in the past two decades (see \cite{VS}, \cite{ASH1,ASH2},\cite{CDL,CDL0}). The triplet $(v,p,\theta)$ on $[0,1]\times{\rm T}^2$ is called a weak solution of (\ref{e:boussinesq equation}) if $\theta\in L^2_{loc}((0,1)\times{\rm T}^2), ~p\in L^2_{loc}((0,1)\times{\rm T}^2),~ \theta\in L^2_{loc}((0,1); H^1({\rm T}^2))$ and solve (\ref{e:boussinesq equation})  in the following sense:
\begin{align}
\int_0^1\int_{{\rm T}^2}(\partial_t\varphi\cdot v+\nabla\varphi:v\otimes v+p {\rm div}\varphi+\theta e_2\cdot\varphi )dxdt=0 \nonumber
\end{align}
for all $\varphi\in C_c^\infty((0,1)\times{\rm T}^2;R^2)$;
\begin{align}
\int_0^1\int_{{\rm T}^2}(\partial_t\phi\theta+v\cdot\nabla\phi\theta-\nabla\theta\cdot \nabla\phi)dxdt=0 \nonumber
\end{align}
for all $\phi\in C_c^\infty((0,1)\times{\rm T}^2;R)$ and
\begin{align}
\int_0^1\int_{{\rm T}^2}v\cdot\nabla\psi dxdt=0 \nonumber
\end{align}
for all $\psi\in C_c^\infty((0,1)\times{\rm T}^2;R).$

The study of constructing non-unique or dissipative weak solution to fluid systems is very fashionable in recent years. The construction is based on the convex integration method pioneered by De Lellis-Sz\'{e}kelyhidi Jr in \cite{CDL, CDL2}, where the authors tackle the Onsager conjecture for the incompressible Euler equation, showing that the incompressible Euler equation admits H\"{o}lder continuous solutions which dissipate kinetic energy.  More precisely, the Onsager conjecture on incompressible Euler equation can be stated as follows:
\begin{enumerate}
  \item $C^{0,\alpha}$ weak solutions are energy conservative when $\alpha>\frac{1}{3}$.
  \item For any $\alpha <\frac{1}{3}$, there exist dissipative solutions with $C^{0,\alpha}$ regularity .
\end{enumerate}

For this conjecture, the part (a) has been proved by P. Constantin, E, Weinan and  E. Titi in \cite{CET}. Slightly weaker assumptions on the solution were subsequently shown to be sufficient for energy conservation by P. Constantin, etc. in \cite{CPFR,DUR}, also see \cite{SH}. More recently, P. Isett and Sung-jin Oh gave a proof to this part of Onsager's conjecture for the Euler equations on manifolds by the heat flow method in \cite{ISOH1}.

Part (b) was proved by P. Isett in \cite{IS2}, based on a series of progress on this problem in \cite{TBU, BCDLI, BCDL1, BCDL2, CHO, DA, DAL, CDL3, ISOH2}, see also \cite{BCDLV} for the construction  of admissible weak solutions. Moreover, the idea and method can be used to construct dissipative weak solutions for other models, see \cite{BSV, ISV2, LX, Nov, RSH, TZ,TZ1,TZ2}.

 Recently, Buckmaster and Vicol established the non-uniqueness of weak solution to the 3D incompressible Navier-Stokes in \cite{BV} by introducing some new ideas. Furthermore, the method was used to construct weak solution in \cite{BCV, LT, LTZ, L, MS0, MS1, MS2}.

Motivated by the study of Onsager's conjecture of the incompressible Euler equation in earlier works and non-uniqueness of weak solutions to the Navier-Stokes equation, we consider the Boussinesq equation with diffusive temperature and want to know if the anomalous dissipation of kinetic energy can also happened when considering the temperature effects. The difference is that there are conversions between internal energy and mechanical energy, and we need to overcome the difficulty of interactions between velocity and temperature. To this end, we consider the existence of H\"{o}lder continuous periodic solutions of the 2d Boussinesq equations with diffusive temperature which satisfy the prescribed kinetic energy. Following the general scheme in the construction of H\"{o}lder continuous weak solution of the incompressible Euler equation and introducing some new ideas, we obtain the following result:

\begin{theorem}\label{t:main 1}
Assume that $e(t):[0,1]\rightarrow R$ is a given positive smooth function and let $\theta^0(x)\in X:= {\rm span}\{\sin(x_2), \cos(x_2)\}$. Then there exist
\begin{align}
    (v,p)\in C([0,1]\times{\rm T}^2), \quad \theta\in C([0,1]\times{\rm T}^2)\cap L^2(0,1; H^1({\rm T}^2))\quad \nonumber
\end{align}
such that they
solve the system \eqref{e:boussinesq equation} in the sense of distribution and
\begin{align}
e(t)=\int_{{\rm T}^2}|v|^{2}(t,x)dx, \quad \forall t\in [0,1].\nonumber
\end{align}
Furthermore, for any $\varepsilon\in (0,\frac{1}{10})$, there holds
\begin{align}
v\in C^{\frac{1}{10}-\varepsilon}_{t,x}, \quad p\in C^{\frac{1}{5}-\varepsilon}_{t,x}, \quad \theta \in C^{1,\frac{1}{20}-\frac{\varepsilon}{2}}_tC^{2,\frac{1}{10}-\varepsilon}_x. \nonumber
\end{align}
Moreover, $\theta$ satisfies the following identity:
\begin{align}
\frac12\|\theta(t,\cdot)\|_{L^2}^2+\int_0^t\|\nabla \theta(s,\cdot)\|^2_{L^2}ds=\frac12\|\theta^0(\cdot)\|_{L^2}^2, \quad \forall t\in [0,1]. \nonumber
\end{align}
\end{theorem}


\begin{Remark}
We consider the following 2d Boussinesq-equation with fractional dissipation in velocity:
\begin{equation}\label{e:boussinesq equation with factional diffusion}
\begin{cases}
\partial_tv+v\cdot\nabla v+\nabla p+(-\triangle)^\alpha v=\theta e_2,\quad\quad \mbox{in}\quad {\rm T}^2\times [0,1]\\
{\rm div}v=0,\\
\partial_t\theta+v\cdot \nabla\theta-\triangle \theta=0, \quad\quad \mbox{in}\quad {\rm T}^2\times [0,1].
\end{cases}
\end{equation}
For $\alpha< \frac12$, we also can construct H\"{o}lder continuous weak solutions for (\ref{e:boussinesq equation with factional diffusion}) by the argument in this paper. For $\frac12\leq \alpha < 1$, we have constructed finite energy weak solutions for (\ref{e:boussinesq equation with factional diffusion}) in \cite{LTZ}.
\end{Remark}


\setcounter{equation}{0}

\section{Proof of main result and Plan of the paper}

As in \cite{BCDLI}, the proof of Theorem \ref{t:main 1} will be achieved through an iteration procedure. In this and subsequent sections, $\mathcal{S}_0^{2\times2}$  denotes the vector space of
trace-free symmetric $2\times2$ matrices.
\begin{Definition}
Assume $(v,p,\theta,\mathring{R})$ are smooth functions on $[0,1]\times {\rm T}^2$ taking values,~respectively,~in $R^2,R,R,\mathcal{S}_{0}^{2\times2}$.
We say that they solve the Boussinesq-Reynold system if
\begin{equation}\label{e:boussinesq-reynold equation}
\begin{cases}
\partial_tv+v\cdot\nabla v+\nabla p=\theta e_2+{\rm div}\mathring{R},\\
{\rm div}v=0,\\
\partial_t\theta+v\cdot \nabla\theta-\triangle\theta=0,\\
\theta(0,x)=\theta^0(x),
\end{cases}
\end{equation}
where $\theta^0(x)\in X$.
\end{Definition}
We now state the main proposition of this paper, by which Theorem \ref{t:main 1} is implied directly. For the statement of the main proposition,  we introduce some notations.

In the following, $m=0,1,\cdot\cdot\cdot$, $\alpha\in (0,1)$ and $\beta$ is a multi-index. We define the H\"{o}lder semi-norms as
\begin{align}
[f]_{m+\alpha}=\max_{|\beta|=m} \sup_{x\neq y,t} \frac{|\nabla^\beta f(t,x)-\nabla^\beta f(t,y)|}{|x-y|^\alpha}\nonumber
\end{align}
and semi-norms
\begin{align}
[f]_m=\sum_{|\beta|=m}\|\nabla^\beta f\|_0 , \nonumber
\end{align}
where \begin{align}
\|v\|_0:=\sup_{t,x}|v(t,x)|, \quad \|v(t,\cdot)\|_0:=\sup_{x}|v(t,\cdot)|, \nonumber
\end{align}
and $\nabla^\beta$ are spatial derivatives. Then, define norms
\begin{align}
\|f\|_m=\sum_{k\leq m}[f]_k, \quad \|f(t,\cdot)\|_m=\sum_{k\leq m}[f(t,\cdot)]_k \nonumber
\end{align}
and H\"{o}lder norms
\begin{align}
\|f\|_{m+\alpha}=\|f\|_m+[f]_{m+\alpha}. \nonumber
\end{align}

Moreover, we introduce two parameters:
\begin{align}
\delta_q=a^{-b^q}, \quad a^{cb^{q+1}}\leq \lambda_q\leq 2a^{cb^{q+1}}, \nonumber
\end{align}
where $a,b,c>1$ and $a$ is a integer.

\begin{proposition}\label{p: iterative 1}
Let $e(t), \theta^0(x)$ be as in Theorem \ref{t:main 1}. Then we can choose two positive constants $\eta$ and $M$ only dependent of $e(t)$ such that the following holds.
For any small $\gamma> 0$ and $b=\frac{6+\gamma}{4}, c=\frac{4(5+\gamma)}{6+\gamma}$, if $a$ is sufficiently large, then there exists a sequence of functions $(v_q,~ p_q, ~\theta_q, ~\mathring{R}_q)\in C^{\infty}([0,1]\times {\rm T}^2)$ starting with $(v_0,~ p_0, ~\theta_0, ~\mathring{R}_0)$, solving the Boussinesq-Stress system (\ref{e:boussinesq-reynold equation}) and satisfying the following estimates
\begin{align}
    & \|\mathring{R}_q\|_0\leq \eta\delta_{q+1},\quad \|\mathring{R}_q\|_1\leq \eta\delta_{q+1}\lambda_q,  \label{e:Reynold stress estimate}\\[3pt]
    &\|v_{q+1}-v_q\|_0 \leq M\delta_{q+1}^{\frac12}, \quad \|v_{q+1}-v_q\|_m \leq C(m)\delta_{q+1}^{\frac12}\lambda^m_{q+1}, \quad m\geq 1 , \label{e:velocity difference estimate} \\[5pt]
    &\|p_{q+1}-p_q\|_0 \leq M^2\delta_{q+1}, \quad  \|p_{q+1}-p_q\|_1 \leq M^2\delta_{q+1}\lambda_{q+1}, \label{e:pressure difference estimate}\\[3pt]
    &\frac12\|\theta_q(t,\cdot)\|_{L^2}^2+\int_0^t\|\nabla\theta_q(s,\cdot)\|_{L^2}^2ds=\frac12 \|\theta^0(\cdot)\|_{L^2}^2, \quad \forall t\in [0,1],\label{e:energy identity}\\[3pt]
    &\|(\theta_{q+1}-\theta_q)(t,\cdot)\|_{L^2}^2+\int_0^t\|\nabla(\theta_{q+1}-\theta_q)(s,\cdot)\|^2_{L^2}ds\leq 4M^2\|\theta^0\|^2_0 \delta_{q+1}, \quad \forall t\in [0,1],\label{e:energy estimate on temperature}\\[3pt]
    &\Big| e(t)(1-\delta_{q+1})-\int_{{\rm T}^2}|v_q|^{2}(x,t)dx\Big|\leq\frac{\delta_{q+1}}{4}e(t),\quad \forall t \in [0,1].
    \label{e:energy difference estimate}
\end{align}
Moreover, there holds
\begin{align}\label{e:time derivative estimate}
\|\partial_t(v_{q+1}-v_q)\|_0\leq C\delta_{q+1}^{\frac12}\lambda_{q+1}, \quad \|\partial_t(p_{q+1}-p_q)\|_0\leq C\delta_{q+1}\lambda_{q+1}.
\end{align}
\end{proposition}
We will prove Proposition \ref{p: iterative 1} in the subsequent sections. Here we first give a proof of Theorem \ref{t:main 1} by using this proposition.
\begin{proof}[Proof of Theorem 1.1]
From (\ref{e:Reynold stress estimate})-(\ref{e:energy estimate on temperature}), we know that $(v_{n},~p_{n},~\mathring{R}_{n})$ are Cauchy sequences in $C([0,1]\times{\rm T}^2)$, and $\theta_n$ is a Cauchy sequence in $L^\infty(0,1; L^2({\rm T}^2))\cap L^2(0,1;H^1({\rm T}^2))$, therefore there exist
\begin{align}
    (v, p)\in C([0,1]\times{\rm T}^2), \quad \theta \in L^\infty(0,1;L^2({\rm T}^2))\cap L^2(0,1; H^1({\rm T}^2))\nonumber
\end{align}
such that
\begin{align}
   (v_{n}, p_n)\rightarrow (v,p) &\quad {\rm in} \quad C([0,1]\times{\rm T}^2),\nonumber\\
    \theta_{n}\rightarrow \theta &\quad {\rm in} \quad L^\infty(0,1;L^2({\rm T}^2))\cap L^2(0,1; H^1({\rm T}^2)), \nonumber\\
   \mathring{R_{n}}\rightarrow 0 &\quad {\rm in} \quad C([0,1]\times{\rm T}^2)\nonumber
\end{align}
as $n \rightarrow \infty$.\\
Passing into limit in (\ref{e:boussinesq-reynold equation}), we conclude that $(v, p, \theta)$ solve (\ref{e:boussinesq equation}) in the sense of distribution.
Moreover, (\ref{e:energy difference estimate}) implies
\begin{align}
e(t)=\int_{{\rm T}^2}|v|^{2}(t,x)dx, \qquad \forall t\in[0,1].\nonumber
\end{align}
From (\ref{e:velocity difference estimate}), (\ref{e:pressure difference estimate}) and interpolation, we conclude
\begin{align}
\|v_{q+1}-v_q\|_\alpha \leq \|v_{q+1}-v_q\|_0^{1-\alpha}\|v_{q+1}-v_q\|_1^{\alpha}\leq M\delta_{q+1}^{\frac12}\lambda_{q+1}^\alpha\leq 2M a^{(-\frac12+\alpha bc)b^{q+1}},\nonumber\\[3pt]
\|p_{q+1}-p_q\|_\alpha \leq \|p_{q+1}-p_q\|_0^{1-\alpha}\|p_{q+1}-p_q\|_1^{\alpha}\leq M^2\delta_{q+1}\lambda_{q+1}^\alpha\leq 2M a^{(-1+\alpha bc)b^{q+1}}.\nonumber
\end{align}
Thus, for any small $\gamma> 0$ and every $\alpha< \frac{1}{2bc}=\frac{1}{2(5+\gamma)}$, $v_q$ is a Cauchy sequence in $C_{t,x}^\alpha$ and $p_q$ is a Cauchy sequence in $C_{t,x}^{2\alpha}$. Hence, for every $\alpha< \frac{1}{10}$, there eixts $(v, p, \theta)$  with $v\in C^\alpha_{t,x},~ p\in C^{2\alpha}_{t,x}$ and they solve the Boussinesq equation. By the Schauder estimate of linear parabolic equations, we deduce that $\theta\in C^{1,\frac{\alpha}{2}}_tC^{2,\alpha}_{x}$.

Furthermore, by (\ref{e:energy identity}), we deduce that the temperature $\theta$ satisfies the energy equality: for every $t\in [0, 1]$
\beno
\|\theta(t,\cdot)\|^2_{L^2}+2\int_0^t \|\nabla\theta(s,\cdot)\|^2_{L^2}=\|\theta^0(\cdot)\|^2_{L^2}.
\eeno
This completes the proof of Theorem \ref{t:main 1}.
\end{proof}

\subsection{Outline of the proof of Propositions \ref{p: iterative 1}}
The rest of this paper will be dedicated to prove Proposition \ref{p: iterative 1}. We perform a inductive procedure, and construct
 $v_{q+1}$ from $v_q$ by adding
some perturbations as follows:
 \begin{align}
 v_{q+1}=&v_q+w_o+w_c:=v_q+w,\nonumber
 \end{align}
where $w_o, w_c$ are smooth functions given by explicit formulas which depend on $(e(t), v_q, \mathring{R}_q)$.
After the construction of the new velocity $v_{q+1}$, we construct the new temperature $\theta_{q+1}$ by solving the following transport-diffusion equation: there exists a $\theta_{q+1}\in C^\infty([0, 1]\times {\rm T}^2, R)$ which solves
\begin{equation}\label{e:transport-diffusion equation}
\begin{cases}
\partial_t \theta_{q+1}+ v_{q+1}\cdot\nabla  \theta_{q+1} - \triangle  \theta_{q+1}=0,\\[3pt]
\theta_{q+1}|_{t=0}=\theta^0,
\end{cases}
\end{equation}
where $\theta^0$ is the function appeared in Proposition \ref{p: iterative 1}.
After the construction of $v_{q+1}, \theta_{q+1}$, we mainly focus on finding functions $\mathring{R}_{q+1}, p_{q+1}$ with the desired estimates and solving the system (\ref{e:boussinesq-reynold equation}).\\

The rest of paper is organized as follows. In Section 3, we do some preliminaries. We introduce the Geometric Lemma in \cite{CDL2}, stationary solutions of 2d Euler equation, anti-divergence operator and estimates of transport-diffusion equation with highly oscillatory forces.
In Section 4, we first define the new velocity $v_{q+1}$ by constructing velocity perturbations $w_o, w_c$, and then construct the new temperature $\theta_{q+1}$ by solving the transport-diffusion equation. In Section 5 and Section 6, we establish various estimates for the perturbation. Finally, in Section 7, we give a proof of Proposition \ref{p: iterative 1} by using the estimate which was established in Section 5 and 6.

\setcounter{equation}{0}

\section{Some preliminaries}

We first recall the following stationary solution for the 2d Euler equation which is the building block in our iterative scheme.
\subsection{Stationary flows in 2D}
\begin{Proposition}\label{Prop:Beltrami-2D}
Let $\Lambda$ be a given finite symmetric subset of $S^1\cap Q^2$. Then for any choice of coefficients $a_{k}\in {\rm C}$ with $\overline{a_k}=a_{-k}$, the vector field
\beno
W(x)=\sum_{k\in \Lambda}a_kik^{\bot}e^{i k\cdot x}, \quad \Psi(x)=\sum_{k\in \Lambda}a_ke^{i k\cdot x}
\eeno
is real-valued and satisfies
\begin{align}\label{i:stationary solution}
{\rm div}(W\otimes W)=\nabla\Big(\frac{|W|^2}{2}+\frac{\Psi^2}{2}\Big), \quad W(x)=\nabla^{\bot}\Psi(x).
\end{align}
Here and throughout the paper, we denote $k^{\perp}=(-k_2, k_1)$ if $k=( k_1, k_2)$, and denote $\nabla^\bot=(-\partial_2, \partial_1)$.
Furthermore,
\beno
\big<W\otimes W\big>:=\fint_{{\rm T}^2}W\otimes W(x)dx=\sum_{k\in \Lambda}|a_{k}|^2({\rm Id}-k\otimes k).
\eeno
\end{Proposition}
The proof of this proposition can be found in \cite{CHO}, and for completeness, we give a direct proof here. We first prove the following lemma.
\begin{Lemma}
Let $\vec{f}_1(x)=\binom{-b}{a}e^{i(a,b)\cdot x}, \quad \vec{f}_2(x)=\binom{-d}{c}e^{i(c,d)\cdot x}$ with $a^2+b^2=c^2+d^2$. There holds
\begin{align}
{\rm div}\big(\vec{f}_1(x)\otimes\vec{f}_2(x)+\vec{f}_2(x)\otimes \vec{f}_1(x)\big)=\nabla \Big((ac+bd-a^2-b^2)e^{i(a+c,b+d)\cdot x}\Big).\nonumber
\end{align}
Here and below, we denote
\begin{align}
\binom{a}{b}\otimes\binom{c}{d}=\binom{a}{b}(c\quad d)=
\left(\begin{array}{ccc}
ac & ad\\
bc & bd
\end{array}\right).\nonumber
\end{align}
\end{Lemma}
\begin{proof}
A computation gives
\begin{align}
&{\rm div}\big(\vec{f}_1(x)\otimes\vec{f}_2(x)+\vec{f}_2(x)\otimes \vec{f}_1(x)\big)\nonumber\\[3pt]
=&{\rm div}\left(\binom{-b}{a}\otimes \binom{-d}{c}e^{i(a+c,b+d)\cdot x}+\binom{-d}{c}\otimes \binom{-b}{a}e^{i(a+c,b+d)\cdot x}\right)\nonumber\\
=&{\rm div}\Big[
\left(\begin{array}{ccc}
2bd & -bc-ad\\
-ad-bc & 2ac
\end{array}\right)e^{i(a+c,b+d)\cdot x}\Big]\nonumber\\
=&i\left(\begin{array}{ccc}
2bd & -bc-ad\\
-ad-bc & 2ac
\end{array}\right)\binom{a+c}{b+d}e^{i(a+c,b+d)\cdot x}\nonumber\\
=&i\binom{bd(a+c)-b^2c-ad^2}{ac(b+d)-ad^2-bc^2}e^{i(a+c,b+d)\cdot x}\nonumber\\[3pt]
=&i\binom{bd(a+c)+ac(a+c)-ac^2-ca^2-b^2c-ad^2}{ac(b+d)+bd(b+d)-bd^2-db^2-ad^2-bc^2}e^{i(a+c,b+d)\cdot x}\nonumber\\[3pt]
=&i\binom{(bd+ac-a^2-b^2)(a+c)}{(ac+bd-a^2-b^2)(b+d)}e^{i(a+c,b+d)\cdot x}\nonumber\\[3pt]
=&\nabla \Big((ac+bd-a^2-b^2)e^{i(a+c,b+d)\cdot x}\Big),\nonumber
\end{align}
where we used $a^2+b^2=c^2+d^2$ in the penultimate line.
\end{proof}

{\bf Proof of Proposition \ref{Prop:Beltrami-2D}}
\begin{proof}
We only prove the first identity in (\ref{i:stationary solution}). The others are obvious.
It's direct to obtain
\begin{align}
W(x)\otimes W(x)&=-\sum_{k,k' \in \Lambda}a_ka_{k'}k^{\bot}\otimes (k')^{\bot} e^{i (k+k')\cdot x}\nonumber\\
&=-\frac12\sum_{k,k' \in \Lambda}a_ka_{k'}(k^{\bot}\otimes (k')^{\bot}+(k')^\bot\otimes k^\bot)e^{i (k+k')\cdot x}, \nonumber\\
\frac{|W(x)|^2}{2}&=-\frac12\sum_{k,k' \in \Lambda}a_ka_{k'}k\cdot k' e^{i (k+k')\cdot x},\nonumber\\
\frac{\Psi^2(x)}{2}&=\frac12\sum_{k,k' \in \Lambda}a_ka_{k'} e^{i (k+k')\cdot x}.\nonumber
\end{align}
Thus, using the above lemma, we obtain
\begin{align}
{\rm div}\big(W(x)\otimes W(x)\big)&=-\frac12\sum_{k,k' \in \Lambda}a_ka_{k'}(k\cdot k'-1)\nabla e^{i (k+k')\cdot x}\nonumber\\
&=-\frac12\nabla \Big(\sum_{k,k' \in \Lambda}a_ka_{k'}(k\cdot k'-1) e^{i (k+k')\cdot x}\Big),\nonumber
\end{align}
\begin{align}
\nabla \Big(\frac{|W(x)|^2}{2}+\frac{\Psi^2(x)}{2}\Big)&=-\frac12\nabla \Big(\sum_{k,k' \in \Lambda}a_ka_{k'}(k\cdot k'-1) e^{i (k+k')\cdot x}\Big),\nonumber
\end{align}
and hence
\beno
{\rm div}(W\otimes W)=\nabla\Big(\frac{|W|^2}{2}+\frac{\Psi^2}{2}\Big).
\eeno
\end{proof}

\subsection{Geometric Lemma}

Let
\begin{align*}
\Lambda_0^+ = \Big\{ e_1,~~ \frac{3}{5}e_1 + \frac{4}{5}e_2, ~~\frac{3}{5}e_1 - \frac{4}{5}e_2  \Big\}, \quad \Lambda_0^- = - \Lambda_0^+, \quad \Lambda_0 = \Lambda_0^+ \cup \Lambda_0^-,
\end{align*}
and $\Lambda_1$ be given by the rotation of $\Lambda_0$ counter clock-wise by $\pi/2$:
\begin{align*}
\Lambda_1^+ = \Big\{ e_2, ~~\frac{3}{5}e_2 + \frac{4}{5}e_1, ~~\frac{3}{5}e_2 - \frac{4}{5}e_1  \Big\}, \quad \Lambda_1^- = - \Lambda_1^+, \quad \Lambda_1 = \Lambda_1^+ \cup \Lambda_1^-.
\end{align*}
Clearly $\Lambda_0, \Lambda_1 \subseteq  Q^2\cap S^1$ and we have the representation
\begin{align*}
\frac{25}{32}\Big(\Big(\frac{3}{5}e_1 + \frac{4}{5}e_2\Big)\otimes& \Big(\frac{3}{5}e_1 +\frac{4}{5}e_2\Big)\nonumber\\
&+\Big( \frac{3}{5}e_1 - \frac{4}{5}e_2\Big)\otimes  \Big(\frac{3}{5}e_1 - \frac{4}{5}e_2\Big)\Big)
 + \frac{7}{16}e_1\otimes e_1 = \mathrm{Id}.
\end{align*}
Moreover, notice that $\mathcal{S}^{2\times 2}$($2\times2$ symmetric matric) is a 3-dimensional linear space, thus by the above representation and uniqueness, we know that such representation holds for $2 \times 2$ symmetric matrices near $\mathrm{Id}$.
\begin{Lemma}[Geometric Lemma]\label{l:geometric lemma}
There exists $\varepsilon_{0}>0$ and smooth positive functions $\gamma_k$:
 \begin{align*}
    \gamma_k\in C^{\infty}(B_{\varepsilon_{0}}({\rm Id})), k \in \Lambda_0^+ 
\end{align*}
such that for every $2 \times 2$ symmetric matrix $R \in B_{\varepsilon_{0}}({\rm Id})$, we have
  $$R = 2\sum_{k \in \Lambda_0^+} \gamma^2_k(R)k \otimes k.$$
\end{Lemma}

\begin{Remark}\label{r:remark about universal constant}
By rotational symmetry, Geometric Lemma \ref{l:geometric lemma} also holds for $k \in \Lambda_1^+$.
It is  convenient to introduce a small geometric constant $c_0\in (0, 1)$ such that
\beno
|k+k'|\geq 2c_0
\eeno
for all $k, k' \in \Lambda_0 \cap \Lambda_1, k\neq -k'$. Moreover, for $k\in \Lambda_l^{-}$ with $l=0,1$, we set $\gamma_k:=\gamma_{-k}$.
\end{Remark}

\subsection{Anti-divergence operator}
We recall the anti-divergence operator in this subsection.
\begin{Lemma}{(Anti-divergence operator)}\label{l:anti-divergence operator}
There exists an operator $\mathcal{R}$ satisfying the following property:
\begin{itemize}
  \item For any $v\in C^\infty({\rm T}^2; R^2)$, $\mathcal{R}v(x)$ is a symmetric trace-free matrix for each $x\in {\rm T}^2$ and
         ${\rm div} \mathcal{R}v(x)= v(x)-\fint_{{\rm T}^2} v(x)dx.$
  \item The following estimates hold: for any $a\in C^\infty({\rm T}^2)$ and any $m\in {\rm N}_+, ~ \alpha\in (0, 1)$, there holds
      \begin{align}\label{e:oscillatory estimate}
      \big\|\mathcal{R}\big(a(x)e^{i\lambda k\cdot x}\big)\big\|_{\alpha}\leq C(m,\alpha)\Big(\frac{\|a\|_0}{\lambda^{1-\alpha}}+\frac{\|\nabla^m a\|_0}{\lambda^{m-\alpha}}+ \frac{\|\nabla^{m}a\|_\alpha}{\lambda^{m}}\Big).
      \end{align}
\end{itemize}
\end{Lemma}
\begin{proof}
Let $u\in C_0^\infty({\rm T}^2)$ be a solution to
\beno
\triangle u=v-\fint_{{\rm T}^2}v(x)dx,
\eeno
where $C_0^\infty({\rm T}^2)=\{f\in C^\infty({\rm T}^2):\fint_{{\rm T}^2} f(x)dx=0\}.$
Then set
\beno
\mathcal{R}v(x):= \nabla u+ (\nabla u)^{{\rm T}} -({\rm div}u) {\rm Id}.
\eeno
Then $\mathcal{R}$ satisfies the above property. The detail can be found in \cite{CHO}, and we omit it here.
\end{proof}

Moreover, $\mathcal{R}$ satisfies the following property:
\begin{Lemma}
For any $s> 0$, there holds
\begin{align}\label{e:L infinity L2 bound for anti-divergence}
\|\mathcal{R}(v)\|_0\leq C \|v\|_{\dot{H}^s}.
\end{align}
Here and below,
\begin{align}
\|v\|_{\dot{H}^s}^2=\sum_{k\in Z^2, k\neq 0} |v_k|^2|k|^{2s}. \nonumber
\end{align}
\end{Lemma}
\begin{proof}
In fact, we can give a explicit formula for $\mathcal{R}$ by the Fourier series expansion. Let
\begin{align}
v(x)=\sum_{k\in Z^2}v_k e^{ik\cdot x}, \quad x\in {\rm T}^2 \nonumber
\end{align}
where $v_k\in {\rm C}^2$ is the Fourier coefficient of the vector function $v(x)$. Then
\begin{align}
v(x)-\fint_{{\rm T}^2}v(x)dx=\sum_{k\in Z^2, k\neq 0}v_k e^{ik\cdot x}, \quad x\in {\rm T}^2. \nonumber
\end{align}
Thus, due to $u\in C_0^\infty({\rm T}^2)$ and $\triangle u=v$, we deduce
\begin{align}
u(x)=\sum_{k\in Z^2, k\neq 0}\frac{-v_k}{|k|^2} e^{ik\cdot x},\quad x\in {\rm T}^2.  \nonumber
\end{align}
Thus, there hold
\begin{align}
\nabla u(x)=&\sum_{k\in Z^2, k\neq 0}\frac{-iv_k\otimes k}{|k|^2} e^{ik\cdot x}, \nonumber\\
(\nabla u)^{\rm T}(x)=&\sum_{k\in Z^2, k\neq 0}\frac{-ik\otimes v_k}{|k|^2} e^{ik\cdot x}, \nonumber\\
{\rm div} u(x)=&\sum_{k\in Z^2, k\neq 0}\frac{-iv_k\cdot k}{|k|^2} e^{ik\cdot x}. \nonumber
\end{align}
Summing them together, we obtain
\begin{align}
\mathcal{R}(v)(x)=\sum_{k\in Z^2, k\neq 0}\Big(\frac{-iv_k\otimes k}{|k|^2}+\frac{-ik\otimes v_k}{|k|^2}+\frac{iv_k\cdot k}{|k|^2}{\rm Id} \Big)e^{ik\cdot x}, \nonumber
\end{align}
hence for any $s>0$, there hold
 \begin{align}
 \|\mathcal{R}(v)\|_0\leq C\sum_{k\in Z^2, k\neq 0}\frac{|v_k|}{|k|}\leq& C\Big(\sum_{k\in Z^2, k\neq 0}|v_k|^2|k|^{2s}\Big)^{\frac12}\Big(\sum_{k\in Z^2, k\neq 0}\frac{1}{|k|^{2(1+s)}}\Big)^{\frac12}\leq C \|v\|_{\dot{H}^s},\nonumber
 \end{align}
 where we used
 \begin{align}
 \sum_{k\in Z^2, k\neq 0}\frac{1}{|k|^{2(1+s)}} \nonumber
 \end{align}
 is convergent for any $s>0$. In fact,
\begin{align}
 \sum_{k\in Z^2, k\neq 0}\frac{1}{|k|^{2(1+s)}}&\leq 4\sum_{n=1}^\infty\sum_{k_1+k_2=n, k_1, k_2\geq 0}\frac{1}{(k_1^2+k_2^2)^{1+s}}\nonumber\\
 &\leq4\sum_{n=1}^\infty\Big(\frac{1}{n^{2(1+s)}}+\frac{1}{(1+(n-1)^2)^{(1+s)}}+\cdot\cdot\cdot+\frac{1}{n^{2(1+s)}}\Big)\nonumber\\
 &\leq C\sum_{n=1}^\infty\frac{2^{(1+s)}(n+1)}{n^{2(1+s)}}\leq C_s< +\infty.\nonumber
 \end{align}
\end{proof}

\subsection{Transport-diffusion equation with oscillatory force} We consider the transport-diffusion equation with oscillatory force
\begin{align}\textcolor{blue}{\label{e:transport-diffusion equation with decay source'}}
\left\{
\begin{array}{ll}
\partial_t \theta+v\cdot \nabla \theta- \triangle \theta=a(t,x)\cos(\lambda k\cdot x), \quad x \in {\rm T}^2, \\[3pt]
{\rm div}v=0, \\[3pt]
\theta(0,x)=0,
\end{array}
\right.
\end{align}
where $k\in S^1\cap Q^2$ is a vector and $\lambda k\in Z^2, \lambda\neq 0$.

To simplify the formulas,  we introduce the following notation: for any $n\in {\rm N}_+$,
\begin{align}
\|(f_1,\cdot\cdot\cdot,f_n)\|_X:=\sum_{k=1}^n \|f_k\|_X, \nonumber
\end{align}
where $\|\cdot\|_X$ is a norm.
\begin{Lemma}
Let $\theta(t,x)$ be a solution of (\ref{e:transport-diffusion equation with decay source'}). Then there hold
\begin{align}\label{e:transport difusion estimate}
\|\theta\|_{L^\infty L^2_x}\leq& C(M)\frac{\| a \|_{L^\infty L^2}}{\lambda}
+ C(M)\frac{\|\nabla^{M+1}a\|_{L^\infty L^2}}{\lambda^{M+1}},\quad \forall M\geq 1, \nonumber\\
\|\nabla^N\theta\|_{L^\infty L^2}
	\leq&  C(N)\|\theta\|_{L^\infty L^2}\sum_{k=0}^{N-1}\|\nabla^k v\|_0^{\frac{N+1}{k+1}}\nonumber\\[3pt]
& \quad \quad \quad +C(N)\|(\nabla^{N-1}a, \lambda^{N-1}a)\|_{L^\infty L^2}, \quad N\geq 1.
\end{align}
\end{Lemma}
\begin{proof}
{\bf Estimate on $\|\theta\|_{L^\infty L^2}$:}
 A direct energy estimate gives
\begin{align}
\frac12\frac{d}{dt}\| \theta(t,\cdot)\|_{L^2}^2 +\|\nabla \theta(t,\cdot)\|_{L^2}^2
	\leq  \Big|\int_{{\rm T}^2}a(t,x)\cos(\lambda k\cdot x)\theta(t,x)dx\Big|.\nonumber
\end{align}
By integration by parts, for any integer $M$, there holds
\begin{align}
&\Big|\int_{{\rm T}^2}a(t,x)\cos(\lambda k\cdot x)\theta(t,x)dx\Big|\nonumber\\
\leq& \frac{\|a(t,\cdot)\|_{L^2}\|\nabla\theta(t,\cdot)\|_{L^2}}{\lambda}
+\frac{1}{\lambda}\Big|\int_{{\rm T}^2}\nabla a(t,x)\sin(\lambda k\cdot x)\theta(t,x)dx\Big| \nonumber\\
\leq & \sum_{i=0}^M \frac{\|\nabla^ia(t,\cdot)\|_{L^2}\|\nabla\theta(t,\cdot)\|_{L^2}}{\lambda^{i+1}}
+\frac{\|\nabla^{M+1}a(t,\cdot)\|_{L^2}\|\theta(t,\cdot)\|_{L^2}}{\lambda^{M+1}}\nonumber\\
\leq & \|\nabla \theta(t,\cdot)\|_{L^2}^2+ \sum_{i=0}^M \frac{\|\nabla^ia(t,\cdot)\|_{L^2}^2}{\lambda^{2(i+1)}}+\frac{\|\nabla^{M+1}a(t,\cdot)\|_{L^2}^2}{\lambda^{2(M+1)}}
+\|\theta(t,\cdot)\|^2_{L^2},\nonumber
\end{align}
thus, we deduce that
\begin{align}
\|\theta\|_{L^\infty L^2_x}\leq& C\sum_{i=0}^M \frac{\|\nabla^i a \|_{L^\infty L^2}}{\lambda^{i+1}}
+ C\frac{\|\nabla^{M+1}a\|_{L^\infty L^2}}{\lambda^{M+1}}.\nonumber
\end{align}
Finally, using interpolation inequality and Young inequality, we know
\begin{align}
\frac{\|\nabla^i a \|_{L^\infty L^2}}{\lambda^{i+1}}\leq \frac{\|\nabla^{M+1} a \|_{L^\infty L^2}}{\lambda^{M+2}}+ \frac{\|a\|_{L^\infty L^2}}{\lambda}, \nonumber
\end{align}
which gives the estimate on $\|\theta\|_{L^\infty L^2_x}$ in (\ref{e:transport difusion estimate}).

{\bf Estimate on $\|\theta\|_{L^\infty \dot{H}^N}$($N\geq 1$):}
For $N\geq 1$, acting $\nabla^N $ on both sides of (\ref{e:transport-diffusion equation with decay source'}), taking $L^2$ inner products with $\nabla^N \theta$ and integrating by parts, we obtain
\begin{align}\label{e:energy estimate n-derivative}
&\frac12 \frac{d}{dt}\|\nabla^N\theta(t,\cdot)\|_{L^2}^2+ \|\nabla^{N+1}\theta(t,\cdot)\|_{L^2}^2\nonumber\\[5pt]
\leq & |\big<\nabla^{N-1}(v\cdot \nabla \theta, \nabla^{N+1}\theta)\big>|+|\big<\nabla^{N-1}(a(t,\cdot)\cos(\lambda k\cdot )), \nabla^{N+1}\theta)\big>|\nonumber\\[3pt]
	\leq & \sum_{k=0}^{N-1}C_{N-1}^k\|\nabla^k v\|_0 \|\nabla^{N-k}\theta(t,\cdot)\|_{L^2}\|\nabla^{N+1}\theta(t,\cdot)\|_{L^2} \nonumber\\[3pt]
&+ \frac12\|\nabla^{N+1}\theta(t,\cdot)\|_{L^2}^2+\|\nabla^{N-1}(a(t,\cdot)\cos(\lambda k\cdot ))\|_{L^2}^2.
\end{align}
Using interpolation inequality
\begin{align}
\|\nabla^m\theta\|_{L^2}\leq \| \theta\|^{\frac{N+1-m}{N+1}}_{L^2}\|\nabla^{N+1}\theta\|^{\frac{m}{N+1}}_{L^2}, \quad 0\leq m\leq N, \nonumber
\end{align}
and Young inequality, we deduce that
\begin{align}
&\sum_{k=0}^{N-1}C_{N-1}^k\|\nabla^k v\|_0 \|\nabla^{N-k}\theta(t,\cdot)\|_{L^2}\|\nabla^{N+1}\theta(t,\cdot)\|_{L^2}\nonumber\\
\leq &\sum_{k=0}^{N-1}C_{N-1}^k\|\nabla^k v\|_0 \|\theta(t,\cdot)\|^{\frac{k+1}{N+1}}_{L^2}\|\nabla^{N+1}\theta(t,\cdot)\|_{L^2}^{\frac{2N+1-k}{N+1}}\nonumber\\
\leq& \frac12\|\nabla^{N+1}\theta(t,\cdot)\|_{L^2}^2+ C(N)\sum_{k=0}^{N-1}\|\nabla^k v\|_0^{\frac{2(N+1)}{k+1}} \|\theta(t,\cdot)\|_{L^2}^2. \nonumber
\end{align}
Moreover, it's easy to obtain
\begin{align}
\|\nabla^{N-1}(a(t,\cdot)\cos(\lambda k\cdot ))\|_{L^2}\leq C(N)\|(\nabla^{N-1}a(t,\cdot), \lambda^{N-1}a(t,\cdot))\|_{L^2}. \nonumber
\end{align}
Putting these estimates into (\ref{e:energy estimate n-derivative}), we get
\begin{align}
\frac12 \frac{d}{dt}\|\nabla^N\theta(t,\cdot)\|_{L^2}^2
	\leq & C(N)\sum_{k=0}^{N-1}\|\nabla^k v\|_0^{\frac{2(N+1)}{k+1}} \|\theta(t,\cdot)\|_{L^2}^2\nonumber\\
&+C(N)\|(\nabla^{N-1}a(t,\cdot), \lambda^{N-1}a(t,\cdot))\|_{L^2}^2,\nonumber
\end{align}
thus, we deduce that
\begin{align}
\|\nabla^N\theta\|_{L^\infty L^2}
	\leq  C(N)\|\theta\|_{L^\infty L^2}\sum_{k=0}^{N-1}\|\nabla^k v\|_0^{\frac{N+1}{k+1}}+
C(N)\|(\nabla^{N-1}a, \lambda^{N-1}a)\|_{L^\infty L^2},\nonumber
\end{align}
this completes the proof of this lemma.
\end{proof}

\begin{Remark}
Similarly, we consider
\begin{align}\label{e:transport-diffusion equation with decay source 1}
\left\{
\begin{array}{ll}
\partial_t \theta+v\cdot \nabla \theta- \triangle \theta=a(t,x)\sin(\lambda k\cdot x), \quad x \in {\rm T}^2, \\[3pt]
{\rm div}v=0, \\[3pt]
\theta(0,x)=0,
\end{array}
\right.
\end{align}
where $k, \lambda$ are as above. Then, the following fact still holds:
Let $\theta(t,x)$ be a solution of (\ref{e:transport-diffusion equation with decay source 1}), then the estimate (\ref{e:transport difusion estimate}) is valid for $\theta$.

Generally, if we consider the transport-diffusion equation
\begin{align}\label{e:transport-diffusion equation with decay source}
\left\{
\begin{array}{ll}
\partial_t \theta+v\cdot \nabla \theta- \triangle \theta=\sum_{|k|=1}a_k(t,x)e^{i\lambda k\cdot x}, \\[3pt]
{\rm div}v=0, \\[3pt]
\theta(0,x)=0,
\end{array}
\right.
\end{align}
where $a_k(t,x)$ complex-valued functions  with $\bar{a}_k(t,x)=a_{-k}(t,x)$. Then, there exists a unique solution $\theta(t,x)$ for (\ref{e:transport-diffusion equation with decay source}) and it satisfies the following estimate
\begin{align}\label{e:estimate on general source term}
\|\theta\|_{L^\infty L^2_x}\leq& C(M)\sum_{|k|=1} \frac{\| a_k \|_{L^\infty L^2}}{\lambda}
+ C(M)\sum_{|k|=1}\frac{\|\nabla^{M+1}a_k\|_{L^\infty L^2}}{\lambda^{M+1}},\quad \forall M\geq 1,\nonumber\\
\|\nabla^N\theta\|_{L^\infty L^2}
	\leq&  C(N)\|\theta\|_{L^\infty L^2}\sum_{i=0}^{N-1}\|\nabla^i v\|_0^{\frac{N+1}{i+1}}\nonumber\\
&+C(N)\sum_{|k|=1}\|(\nabla^{N-1}a_k, \lambda^{N-1}a_k)\|_{L^\infty L^2}, \quad N\geq 1.
\end{align}
\end{Remark}

\setcounter{equation}{0}

\section{Construction of $(v_{q+1}, p_{q+1}, \theta_{q+1}, \mathring{R}_{q+1})$}

In this section, we perform the inductive procedure which allows us to construct $(v_{q+1}, p_{q+1}, \theta_{q+1}, \\
\mathring{R}_{q+1})$ from $(v_{q}, p_{q}, \theta_{q}, \mathring{R}_{q})$. Recalling the choice of the sequence $\{\delta_q\}_{q\in N}$ and $\{\lambda_q\}_{q\in N}$,  for sufficiently large $a$, we have, for any $q\geq 1$,
\begin{align}
\sum_{j\leq q}\delta_j \lambda_j^k \leq 2 \delta_q\lambda_q^k,  \quad  \sum_{j\leq q}\delta_j^{\frac12} \lambda_j^k \leq 2 \delta_q^{\frac12}\lambda_q^k,  \quad k\geq 1. \nonumber
\end{align}

 As in \cite{BCDLI}, we write $(v, p, \theta, \mathring{R})$ instead of $(v_{q}, p_{q}, \theta_{q}, \mathring{R}_{q})$ and $(v_1, p_1, \theta_1, \mathring{R}_1)$ instead of $(v_{q+1}, p_{q+1}, \theta_{q+1},\\
  \mathring{R}_{q+1})$. Thus, the following estimates hold:
 \begin{align}
 \|v\|_0\leq& 2M, \quad \|v\|_m\leq C(m) \delta_q^{\frac12}\lambda_q^m, \quad m\geq 1, \label{e: estimate on known velocity}\\[3pt]
 \|\mathring{R}\|_0\leq&\eta \delta_{q+1} , \quad \|\mathring{R}\|_1\leq M \delta_{q+1}\lambda_q.  \label{e: estimate on known Reynold stress}
 \end{align}

\subsection{Space-time regularization of $v,\mathring{R}$} Let $\psi\in C_c^\infty(R^3\times R)$ be a radial symmetry nonnegative function supported in $[-1,1]^4$ and $\ell$ be a small parameter. Set
\begin{align}
v_{\ell}=v\ast \psi_{\ell},  \quad \mathring{R}_{\ell}=\mathring{R}\ast \psi_{\ell}. \nonumber
\end{align}
Standard estimates on convolutions give
\begin{align}\label{e:estimate on convolution 1}
\|v-v_{\ell}\|_0\leq CM\delta_q^{\frac12}\lambda_q \ell, \quad   \|\mathring{R}-\mathring{R}_{\ell}\|_0\leq CM\delta_{q+1}\lambda_q\ell,
\end{align}
and for any $N\geq 1$, there exists a constant $C=C(N)$ such that
\begin{align}\label{e:estimate on convolution 2}
\|v_{\ell}\|_{N}\leq CM\delta_q^{\frac12}\lambda_q \ell^{1-N}, \quad \|\mathring{R}_{\ell}\|_{N}\leq CM\delta_{q+1}\lambda_q \ell^{1-N}.
\end{align}

\subsection{Partition of unity on time} Fix a smooth function $\chi\in C_c^\infty ((-\frac34, \frac34))$ such that
\begin{align}
\sum_{l\in Z}\chi^2(x-l)=1 \nonumber
\end{align}
and a large parameter $\mu \in {\rm N}_+$ which will be determined later.

For any $l\in [0, \mu]$, set
\begin{align}\label{d:definition of l amplitude}
\rho_l:=\frac{1}{2(2\pi)^2}\Big[e\Big(\frac{l}{\mu}\Big)\big(1-\delta_{q+2}\big)-\int_{{\rm T}^2}|v|^2\Big(x, \frac{l}{\mu}\Big)dx\Big].
\end{align}
Due to (\ref{e:energy difference estimate}), we deduce that there exists a universal constant $C_0$ such that
\begin{align}\label{e:estimate on amplitude}
C_0^{-1} \min_{t\in [0,1]}e(t)\delta_{q+1}\leq\rho_l\leq C_0 \min_{t\in [0,1]}e(t)\delta_{q+1}.
\end{align}

\subsection{Construction of velocity perturbation $w$}

Firstly, for any integer $l\in [0,\mu]$, we construct a smooth functions $\Phi_l(t,x): {\rm T}^2\times [0,1]\rightarrow {\rm T}^2$ by solving the following transport equations:
\begin{align}\label{e:transport phase equation}
\left\{
\begin{array}{ll}
\partial_t \Phi_l+ v_{\ell}\cdot \nabla \Phi_l=0, \\[3pt]
\Phi_l(\frac{l}{\mu},x)=x.
\end{array}
\right.
\end{align}

Next, for any $k\in \Lambda_0\cup \Lambda_1$ and any integer $l\in [0, \mu]$, we set
\begin{align}\label{d:difinition on amplitude}
\chi_l(t):=&\chi(\mu t-l), \nonumber\\
a_{kl}(t,x):=&\sqrt{\rho_l}\gamma_k\Big(\frac{R_{\ell,l}(t,x)}{\rho_l}\Big),\nonumber\\
w_{kl}(t,x):=& a_{kl}(t,x)ike^{i\lambda_{q+1}k^\bot\cdot \Phi_l(t,x)}\textcolor{blue}{.}
\end{align}
Here and throughout the paper, $R_{\ell,l}(t,x):=\rho_l {\rm Id}- \mathring{R}_{\ell}(t,x)$.

Due to (\ref{e: estimate on known Reynold stress}) and (\ref{e:estimate on amplitude}), we deduce that there exists an $\eta=\eta(e, r_0):=C_0^{-1}r_0\min_{t\in [0,1]}e(t)$  such that
\begin{align}
\Big\|{\rm Id}-\frac{R_{\ell,l}}{\rho_l}\Big\|_0\leq \frac{r_0}{2}\textcolor{blue}{.}\nonumber
\end{align}
Hence $a_{kl}$ in (\ref{d:difinition on amplitude}) is well-defined.

We define the principle part $w_o$ of the velocity perturbation $w$
\begin{align}\label{d:definition of main perturbation}
w_o(t,x):= \sum_{l\in {\rm Z}}\sum_{k\in \Lambda_{(l)}} \chi_l(t)w_{kl}(t,x),
\end{align}
where
\begin{align}
\Lambda_{(l)}=\Lambda_{l~ mod ~ 2}.\nonumber
\end{align}
Moreover, as in \cite{BCDLI}, we set
\begin{align}\label{e:difinition of stationary phase}
\phi_{kl}(t,x):=e^{i\lambda_{q+1} k^\perp\cdot (\Phi_l(t,x)-x)}.
\end{align}
Then
\begin{align}\label{e:another formulation for main perturbation}
w_o(t,x)= \sum_{l\in {\rm Z}}\sum_{k\in \Lambda_{(l)}} \chi_l(t)a_{kl}ik \phi_{kl}e^{i\lambda_{q+1} k^\perp\cdot x}.
\end{align}

Then, set the incompressibility corrector $w_c$ as
\begin{align}\label{d:difinition of corrctor}
w_c(t,x):=&  -\sum_{l\in {\rm Z}}\sum_{k\in \Lambda_{(l)}}\chi_l(t) \frac{\nabla^{\bot}(a_{kl}(t,x)e^{i\lambda_{q+1} k^\perp\cdot (\Phi_l(t,x)-x)})}{\lambda_{q+1}}e^{i\lambda_{q+1} k^\perp\cdot x}\nonumber\\
:=& -\sum_{l\in {\rm Z}}\sum_{k\in \Lambda_{(l)}}\chi_l(t)\Big( \frac{\nabla^{\bot}a_{kl}(t,x)+ia_{kl}\lambda_{q+1}(\nabla^\bot \Phi_l(t,x)-{\rm \widetilde{Id}})k^\perp}{\lambda_{q+1}}\Big)e^{i\lambda_{q+1} k^\perp\cdot \Phi_l(t,x)}\textcolor{blue}{.}
\end{align}
Here and throughout the paper, $\widetilde{{\rm Id}}$ denotes
\begin{align}
\widetilde{{\rm Id}}=\left(\begin{array}{ccc}
0 & -1\\
1 & 0
\end{array}\right).\nonumber
\end{align}

Finally, we set
\begin{align}
w(t,x):=w_o(t,x)+w_c(t,x).
\end{align}
Obviously,
\begin{align}
w(t,x)=\sum_{l\in {\rm Z}}\sum_{k\in \Lambda_{(l)}}\chi_l(t) \frac{\nabla^{\bot}(a_{kl}(t,x)e^{i\lambda_{q+1} k^\perp\cdot \Phi_l(t,x)})}{\lambda_{q+1}}, \nonumber
\end{align}
hence ${\rm div}w(t,x)=0.$

Set
\begin{align}\label{d:difinition on L_{kl}}
L_{kl}(t,x):=a_{kl}(t,x)ik-\frac{\nabla^{\bot}a_{kl}(t,x)+ia_{kl}\lambda_{q+1}(\nabla^\bot \Phi_l(t,x)-\widetilde{{\rm Id}})k^\perp}{\lambda_{q+1}},
\end{align}
thus the perturbation $w$ can also be written as
\begin{align}\label{f:another formulation for full perturbation}
w=\sum_{l\in {\rm Z}}\sum_{k\in \Lambda_{(l)}}\chi_l(t) L_{kl}\phi_{kl}e^{i\lambda_{q+1} k^\perp\cdot x}.
\end{align}

After the construction of the perturbation $w$, we choose the constant $M$. By (\ref{e:estimate on amplitude}) and the support property of $\chi_l(t)$, it's direct to obtain
\begin{align}
\|w_o\|_0\leq  C_0(e)\delta_{q+1}^{\frac12}. \nonumber
\end{align}
Then, we set
\begin{align}
M=2 C_0(e); \nonumber
\end{align}
thus, there holds \begin{align}\label{e:estimate on main perturbation}
\|w_o\|_0\leq  \frac{M}{2}\delta_{q+1}^{\frac12}.
\end{align}

\subsection{Construction of new velocity $v_1$ and new temperature $\theta_1$} After the construction of the velocity perturbation $w$, we define the new velocity as follows:
\begin{align}
v_1(t,x):=v(t,x)+w(t,x). \nonumber
\end{align}
Thus there holds ${\rm div}v_1(t,x)=0$. Then, we define the new temperature by solving the following transport-diffusion equation
\begin{align}
\left\{
\begin{array}{ll}
\partial_t \theta_1 + v_1\cdot\nabla \theta_1-\triangle \theta_1=0,\\[3pt]
\theta_1(0,x)=\theta^0(x).
\end{array}
\right.
\end{align}
By the maximum principle, we deduce that
\begin{align}\label{e:estimate on temperature 1}
\|\theta_1\|_0\leq \|\theta^0\|_0.
\end{align}
Moreover, the basic energy estimate gives
\begin{align}\label{e:energy conservation of temperature}
\frac12\frac{d}{dt} \|\theta_1(t,\cdot)\|_{L^2}^2+\|\nabla \theta_1(t,\cdot)\|_{L^2}^2=0.
\end{align}
 Finally, it's obvious that
\begin{align}\label{e:vanishing of averaged temperature}
\fint_{{\rm T}^2}\theta_1(t,x)dx=0.
\end{align}

\subsection{The new pressure $p_1$ and new Reynolds stress $\mathring{R}_1$}

We first compute ${\rm div}(w_o\otimes w_o)$. Recalling (\ref{d:definition of main perturbation}), we deduce
\begin{align}\label{c:computation of nonlinear interaction}
w_o\otimes w_o= &\sum_{l\in {\rm Z}}\sum_{k\in \Lambda_{(l)}}\chi_l^2 w_{kl}\otimes w_{-kl}+ \sum_{l \in {\rm Z}}\sum_{k,k'\in \Lambda_{(l)},k+k'\neq 0}\chi_l^2 w_{kl}\otimes w_{k'l}\nonumber\\
&+ \sum_{l\neq l', k\in \Lambda_{(l)}, k'\in \Lambda_{(l')}}\chi_l \chi_{l'}w_{kl}\otimes w_{k'l'}\nonumber\\[3pt]
= & \sum_{l\in {\rm Z}}\chi_l^2 R_{\ell,l}+ \sum_{l \in {\rm Z}}\sum_{k,k'\in \Lambda_{(l)},k+k'\neq 0}\chi_l^2 w_{kl}\otimes w_{k'l}\nonumber\\
&+ \sum_{l\neq l', k\in \Lambda_{(l)}, k'\in \Lambda_{(l')}}\chi_l \chi_{l'}w_{kl}\otimes w_{k'l'},
\end{align}
where we used the following fact
\begin{align}
\sum_{l\in {\rm Z}}\sum_{k\in \Lambda_{(l)}}\chi_l^2 w_{kl}\otimes w_{-kl}=&\sum_{l\in {\rm Z}}\sum_{k\in \Lambda_{(l)}}\chi_l^2 \rho_l\gamma_k^2\Big(\frac{R_{\ell,l}}{\rho_l}\Big)k\otimes k \nonumber\\
=&\sum_{l\in {\rm Z}}\chi_l^2 \rho_l\sum_{k\in \Lambda_{(l)}}\gamma_k^2\Big(\frac{R_{\ell,l}}{\rho_l}\Big)k\otimes k
=\sum_{l\in {\rm Z}}\chi_l^2 R_{\ell,l}. \nonumber
\end{align}
Here we used Geometric Lemma \ref{l:geometric lemma} in the last step.

Furthermore,
\begin{align}\label{c:computation of nonlinear interaction in same index}
&\sum_{l \in {\rm Z}}\sum_{k,k'\in \Lambda_{(l)},k+k'\neq 0}\chi_l^2 w_{kl}\otimes w_{k'l}\nonumber\\
=&\frac12 \sum_{l \in {\rm Z}}\sum_{k,k'\in \Lambda_{(l)},k+k'\neq 0}\chi_l^2 (w_{kl}\otimes w_{k'l}+w_{k'l}\otimes w_{kl})\nonumber\\
=&-\frac12 \sum_{l \in {\rm Z}}\sum_{k,k'\in \Lambda_{(l)},k+k'\neq 0}\chi_l^2 a_{kl}a_{k'l}\phi_{kl}\phi_{k'l}[k \otimes k'+k'\otimes k]e^{i\lambda_{q+1}(k+k')^\perp\cdot x}\textcolor{blue}{;}
\end{align}
thus, by Proposition \ref{Prop:Beltrami-2D}, we deduce that
\begin{align}
&{\rm div}\Big(\sum_{l \in {\rm Z}}\sum_{k,k'\in \Lambda_{(l)},k+k'\neq 0}\chi_l^2 w_{kl}\otimes w_{k'l}\Big)\nonumber\\
=&-\frac12 \sum_{l \in {\rm Z}}\sum_{k,k'\in \Lambda_{(l)},k+k'\neq 0}\chi_l^2[k \otimes k'+k'\otimes k]\nabla(a_{kl}a_{k'l}\phi_{kl}\phi_{k'l})e^{i\lambda_{q+1}(k+k')^\perp\cdot x}\nonumber\\
&-\frac12 \sum_{l \in {\rm Z}}\sum_{k,k'\in \Lambda_{(l)},k+k'\neq 0}\chi_l^2 a_{kl}a_{k'l}\phi_{kl}\phi_{k'l}[k \otimes k'+k'\otimes k]i\lambda_{q+1}(k+k')^\perp e^{i\lambda_{q+1}(k+k')^\perp\cdot x}\nonumber\\
=&T^1_{osc}
-\nabla\Big(\frac12 \sum_{l \in {\rm Z}}\sum_{k,k'\in \Lambda_{(l)},k+k'\neq 0}\chi_l^2 a_{kl}a_{k'l}\phi_{kl}\phi_{k'l}\big(k\cdot k'+1\big)e^{i\lambda_{q+1}(k+k')^\perp\cdot x}\Big), \nonumber
\end{align}
where
\begin{align}\label{e:definition of T1}
T^1_{osc}=&-\frac12 \sum_{l \in {\rm Z}}\sum_{k,k'\in \Lambda_{(l)},k+k'\neq 0}\chi_l^2[k \otimes k'+k'\otimes k]\nabla(a_{kl}a_{k'l}\phi_{kl}\phi_{k'l})e^{i\lambda_{q+1}(k+k')^\perp\cdot x}\nonumber\\
&+\frac12 \sum_{l \in {\rm Z}}\sum_{k,k'\in \Lambda_{(l)},k+k'\neq 0}\chi_l^2\big(k\cdot k'+1\big)\nabla(a_{kl}a_{k'l}\phi_{kl}\phi_{k'l})e^{i\lambda_{q+1}(k+k')^\perp\cdot x}.
\end{align}

Moreover, set
\begin{align}\label{d:difinition on f}
f_{klk'l'}:=\chi_l\chi_{l'}a_{kl}a_{k'l'}\phi_{kl}\phi_{k'l'}\textcolor{blue}{.}
 \end{align}
From (\ref{e:another formulation for main perturbation}) and the fact $f_{klk'l'}=f_{k'l'kl}$, we deduce that
\begin{align}\label{c:computation of nonlinear interaction in different index}
&\sum_{l\neq l', k\in \Lambda_{(l)}, k'\in \Lambda_{(l')}}\chi_l \chi_{l'}w_{kl}\otimes w_{k'l'}\nonumber\\
=& -\sum_{l\neq l', k\in \Lambda_{(l)}, k'\in \Lambda_{(l')}}f_{klk'l'} k \otimes k' e^{i\lambda_{q+1}(k+k')^\perp\cdot x}\nonumber\\
=& -\sum_{l\neq l', k\in \Lambda_{(l)}, k'\in \Lambda_{(l')}}f_{klk'l'} k' \otimes k e^{i\lambda_{q+1}(k+k')^\perp\cdot x}\nonumber\\
=& -\frac12\sum_{l\neq l', k\in \Lambda_{(l)}, k'\in \Lambda_{(l')}}f_{klk'l'} \big(k \otimes k' +k'\otimes k\big)e^{i\lambda_{q+1}(k+k')^\perp\cdot x}.
\end{align}
Thus, by Proposition \ref{Prop:Beltrami-2D}, there hold
\begin{align}
&{\rm div}\Big(\sum_{l\neq l', k\in \Lambda_{(l)}, k'\in \Lambda_{(l')}}\chi_l \chi_{l'}w_{kl}\otimes w_{k'l'}\Big)\nonumber\\
=& -\frac12\sum_{l\neq l', k\in \Lambda_{(l)}, k'\in \Lambda_{(l')}} \big(k \otimes k' +k'\otimes k\big)\nabla f_{klk'l'}
e^{i\lambda_{q+1}(k+k')^\perp\cdot x}\nonumber\\
&-\frac12\sum_{l\neq l', k\in \Lambda_{(l)}, k'\in \Lambda_{(l')}} \big(k\cdot k'+1\big) f_{klk'l'}\nabla e^{i\lambda_{q+1}(k+k')^\perp\cdot x}\nonumber\\
:=&T^2_{osc}-\frac12\nabla \Big[\sum_{l\neq l', k\in \Lambda_{(l)}, k'\in \Lambda_{(l')}} \big(k\cdot k'+1\big) f_{klk'l'}e^{i\lambda_{q+1}(k+k')^\perp\cdot x}\Big],\nonumber
\end{align}
where
\begin{align}\label{d:definiion on second oscillatory}
T^2_{osc}=& -\frac12\sum_{l\neq l', k\in \Lambda_{(l)}, k'\in \Lambda_{(l')}} \big(k\otimes k' +k' \otimes k\big)\nabla f_{klk'l'}e^{i\lambda_{q+1}(k+k')^\perp\cdot x}\nonumber\\
&+\frac12\sum_{l\neq l', k\in \Lambda_{(l)}, k'\in \Lambda_{(l')}} \big(k\cdot k'+1\big) \nabla f_{klk'l'}e^{i\lambda_{q+1}(k+k')^\perp\cdot x}.
\end{align}
Set
\begin{align}\label{e:perturbation on pressure}
P:=&\frac12 \sum_{l \in {\rm Z}}\sum_{k,k'\in \Lambda_{(l)},k+k'\neq 0}\chi_l^2 a_{kl}a_{k'l}\phi_{kl}\phi_{k'l}\big(k\cdot k'+1\big)e^{i\lambda_{q+1}(k+k')^\perp\cdot x}\nonumber\\
&+\frac12\sum_{l\neq l', k\in \Lambda_{(l)}, k'\in \Lambda_{(l')}} \big(k\cdot k'+1\big) f_{klk'l'}e^{i\lambda_{q+1}(k+k')^\perp\cdot x}\textcolor{blue}{;}
\end{align}
thus, by (\ref{c:computation of nonlinear interaction})-(\ref{e:perturbation on pressure}), we deduce
\begin{align}\label{e:formular of oscillatory term}
{\rm div}\Big(w_o\otimes w_o- \sum_l\chi^2_l R_{\ell,l}+ P{\rm Id}\Big)=T^1_{osc}+T^2_{osc}.
\end{align}

After the computation of ${\rm div}(w_o \otimes w_o)$, we set
\begin{align}\label{d:difinition of Reynold stress}
\mathring{R}_1:=& \underbrace{\mathcal{R}(\partial_t w+v_{\ell}\cdot \nabla w)}_{R^0}+ \underbrace{\mathcal{R}( w\cdot\nabla v_{\ell})}_{R^1}
-\underbrace{\mathcal{R}((\theta_1-\theta)e_2)}_{R^2}\nonumber\\
&+\underbrace{\mathcal{R}\Big({\rm div}\Big(w_o\otimes w_o-\sum_l\chi^2_l R_{\ell,l}+P{\rm Id}\Big)\Big)}_{R^3}\nonumber\\
&+\underbrace{w_o\otimes w_c+w_c\otimes w_o+w_c\otimes w_c-\frac{|w_c|^2+2 w_o\cdot w_c}{2}{\rm Id}}_{R^4}\nonumber\\
&+\underbrace{w\otimes(v-v_{\ell})+(v-v_{\ell})\otimes w- (v-v_{\ell})\cdot w{\rm Id}}_{R^5}
+ \underbrace{\mathring{R}-\mathring{R}_{\ell}}_{R^6}.
\end{align}

Obviously, there holds ${\rm tr}\mathring{R}_1=0$. Finally, put
\begin{align}\label{d:definition of new pressure}
p_1=p+P-\frac{|w_c|^2+2 w_o\cdot w_c}{2}-(v-v_{\ell})\cdot w.
\end{align}
A direct computation gives
\begin{align}
\partial_t v_1+v_1\cdot\nabla v_1+\nabla p_1=\theta_1 e_2+{\rm div}\mathring{R}_1. \nonumber
\end{align}
In fact,
\begin{align}
{\rm div}\mathring{R}_1=&\partial_t w+v\cdot \nabla w+ w\cdot\nabla v-(\theta_1-\theta)e_2\nonumber\\[3pt]
&+{\rm div}\Big(w_o\otimes w_o-\sum_l\chi^2_l R_{\ell,l}+P{\rm Id}\Big)+ {\rm div}\Big(\mathring{R}-\mathring{R}_{\ell}\Big)\nonumber\\
&+{\rm div}\Big(w_o\otimes w_c+w_c\otimes w_o+w_c\otimes w_c-\frac{|w_c|^2+2 w_o\cdot w_c}{2}{\rm Id}-(v-v_{\ell})\cdot w {\rm Id}\Big)\nonumber\\[3pt]
=&\partial_t v_1+v_1\cdot\nabla v_1+\nabla p_1-\theta_1 e_2,\nonumber
\end{align}
where we used
\begin{align}
-{\rm div}\Big(\sum_l\chi^2_l R_{\ell,l}\Big)=-{\rm div}\Big(\sum_l\chi^2_l (\rho_l - \mathring{R}_{\ell})\Big)={\rm div} \mathring{R}_{\ell}. \nonumber
\end{align}
Thus, the new functions $(v_1, p_1, \theta_1, \mathring{R}_1)$ solves the Boussinesq-Reynold (\ref{e:boussinesq-reynold equation}) system.

\setcounter{equation}{0}

\section{Estimate on the perturbation}

\subsection{Some elementary inequalities}

Recall the following elementary inequalities:
\begin{align}\label{e:elementary inequality}
[fg]_\alpha\leq& C\big([f]_\alpha \|g\|_0+[g]_\alpha\|f\|_0\big), \quad \forall \alpha \geq 0\nonumber\\
\|f g\|_m\leq& C(m)\big(\|f\|_m \|g\|_0+\|g\|_m\|f\|_0\big), \quad \forall m\in {\rm N}\nonumber\\
[f\circ g]_m\leq & C\big([f]_1[g]_m+\|\nabla f\|_{m-1}\|g\|_1^m\big), \quad m\in {\rm N}_+.
\end{align}

\subsection{Condition on the parameter} To simplify the computation, as in \cite{BCDLI}, we assume the following conditions on $\mu, \ell\leq 1, \lambda_{q+1}\geq 1$:
\begin{align}\label{e:parameter assumption 1}
\frac{\delta_q^{\frac12}\lambda_q \ell}{\delta_{q+1}^{\frac12}}\leq 1, \quad \frac{\delta_q^{\frac12}\lambda_q}{\mu}+\frac{1}{\ell \lambda_{q+1}}\leq \lambda^{-\beta}_{q+1}, \quad \frac{1}{\lambda_{q+1}}\leq \frac{\delta^{\frac12}_{q+1}}{\mu},
\end{align}
where $\beta> 0$ is a number which will be determined later. These conditions imply
\begin{align}\label{e:parameter assumption 2}
\frac{1}{\delta_{q+1}^{\frac12}\lambda_{q+1}}\leq \frac{1}{\mu}\leq \frac{1}{\delta_{q}^{\frac12}\lambda_{q}}, \quad \frac{1}{\lambda_{q+1}}\leq \ell \leq \frac{1}{\lambda_{q}}.
\end{align}

\subsection{Estimate on velocity perturbation}
In this subsection, we collect some estimates on the velocity perturbation.

\begin{Lemma}
Assume (\ref{e:parameter assumption 1}) holds. For any $l\in Z$ and $t$ in the range $|\mu t-l|<1$, we have
\begin{align}\label{e:estimate on transport phase}
\|\nabla \Phi_l\|_0\leq& C, \nonumber\\
 \|\nabla \Phi_l-{\rm Id}\|_0\leq& C\delta_q^{\frac12}\lambda_q\mu^{-1}, \nonumber\\
\|\nabla\Phi_l\|_N\leq& C(N)\delta_q^{\frac12}\lambda_q\mu^{-1} \ell^{-N}, \quad N\geq 1.
\end{align}
Moreover, there hold
\begin{align}\label{e:estimate on various amplitude}
\|a_{kl}\|_0+\|L_{kl}\|_0\leq& C\delta^{\frac12}_{q+1}, \nonumber\\
 \|a_{kl}\|_N\leq& C(N)\delta^{\frac12}_{q+1}\lambda_q \ell^{1-N}, \quad N\geq 1 \nonumber\\
\|L_{kl}\|_N\leq& C(N)\delta^{\frac12}_{q+1} \ell^{-N}, \quad N\geq 1 \nonumber\\
\|(\partial_t+v_{\ell}\cdot\nabla)L_{kl}\|_N\leq& C(N)\delta^{\frac12}_{q+1}\ell^{-N-1}, \quad \forall N\geq 0  \nonumber\\
\|\phi_{kl}\|_N\leq& C(N)\lambda_{q+1}\delta_q^{\frac12}\lambda_q\mu^{-1} \ell^{-N+1}+C\Big(\delta_q^{\frac12}\lambda_q\lambda_{q+1}\mu^{-1}\Big)^N\nonumber\\
\leq& C(N)\lambda_{q+1}^{N(1-\beta)},\quad  N\geq 1.
\end{align}
Consequently, there hold
\begin{align}\label{e:estimate on velocity perturbation}
\|w_c\|_N\leq & C(N)\delta_{q+1}^{\frac12}\delta_q^{\frac12}\lambda_q\mu^{-1}\lambda^N_{q+1}, \quad N\geq 0 \nonumber\\
 \|w_o\|_N\leq&  C(N)\delta_{q+1}^{\frac12}\lambda^N_{q+1}, \quad N\geq 1.
\end{align}
\end{Lemma}
\begin{proof}
{\bf Estimate on $\Phi_l$:} The first estimate in (\ref{e:estimate on transport phase}) can be directly obtained by using (\ref{e:transport estimate for holder derivative}), the second and third estimates in (\ref{e:estimate on transport phase}) can be directly obtained by using (\ref{e:inverse flux estimate}).

{\bf Estimates on $a_{kl}$:} Firstly, by (\ref{e:estimate on amplitude}), it's easy to obtain
\begin{align}
\|a_{kl}\|_0\leq C \delta_{q+1}^{\frac12}. \nonumber
\end{align}

Recalling (\ref{e: estimate on known Reynold stress}), (\ref{e:estimate on convolution 2}), (\ref{e:estimate on amplitude}), (\ref{d:difinition on amplitude}),  (\ref{e:elementary inequality}) and (\ref{e:parameter assumption 2}),  it's easy to obtain
\begin{align}
\|a_{kl}\|_N\leq& C(N)\delta_{q+1}^{\frac12}\Big(\Big\|\frac{\mathring{R}_{\ell}}{\rho_l}\Big\|_N + \Big\|\frac{\mathring{R}_{\ell}}{\rho_l}\Big\|_1^N\Big)\nonumber\\[3pt]
\leq& C(N)\delta_{q+1}^{\frac12}\big(\lambda_q\ell^{-N+1} + \lambda_q^N\big)\leq C(N)\delta_{q+1}^{\frac12}\lambda_q\ell^{-N+1} \nonumber
\end{align}
for any $N\geq 1$.

{\bf Estimates on $L_{kl}$:}
Recalling (\ref{d:difinition on L_{kl}}), by (\ref{e:estimate on transport phase}) and (\ref{e:parameter assumption 1}), we deduce that
\begin{align}
\|L_{kl}\|_0\leq &C\big(\|a_{kl}\|_0+\|a_{kl}\|_{1}\lambda_{q+1}^{-1}+\|a_{kl}(\nabla \Phi_l-{\rm Id})\|_0\big)\nonumber\\
\leq& C\delta_{q+1}^{\frac12}\big(1+\lambda_q \lambda_{q+1}^{-1}+\delta_q^{\frac12}\lambda_q \mu^{-1}\big)\leq C\delta_{q+1}^{\frac12},\nonumber\\
\|L_{kl}\|_N\leq &C(N)\big(\|a_{kl}\|_{N}+\|a_{kl}\|_{N+1}\lambda_{q+1}^{-1}+\|a_{kl}(\nabla \Phi_l-{\rm Id})\|_{N}\big)\nonumber\\
\leq& C(N)\delta_{q+1}^{\frac12}\ell^{-N}\big(\lambda_q \ell+\delta_q^{\frac12}\lambda_q \mu^{-1}\big)\leq C(N)\delta_{q+1}^{\frac12}\ell^{-N}, \quad \forall N\geq 1.\nonumber
\end{align}
Notice that
\begin{align}
(\partial_t+v_{\ell}\cdot \nabla)\Phi_{l}=0, \quad
(\partial_t+v_{\ell}\cdot \nabla)\nabla^\bot \Phi_l=-\sum_{j=1}^2(\nabla^\bot v_{\ell}^j) \partial_j \Phi_{l}.\nonumber
\end{align}
Thus, we obtain
\begin{align}
(\partial_t+v_{\ell}\cdot \nabla)L_{kl}=&(\partial_t+v_{\ell}\cdot \nabla) a_{kl}ik-\frac{(\partial_t+v_{\ell}\cdot \nabla) \nabla^{\bot}a_{kl}}{\lambda_{q+1}}\nonumber\\
&-i(\partial_t+v_{\ell}\cdot \nabla)a_{kl}(\nabla^{\bot}\Phi_l-\widetilde{{\rm Id}})k^{\bot}-ia_{kl}\sum_{j=1}^2(\nabla^\bot v_{\ell}^j) \partial_j \Phi_{l}.\nonumber
\end{align}
Moreover,
\begin{align}
(\partial_t+v_{\ell}\cdot \nabla) a_{kl}=-\frac{1}{\sqrt{\rho_l}} \nabla\gamma_k\Big({\rm Id}-\frac{\mathring{R}_{\ell}}{\rho_l}\Big) \cdot (\partial_t+v_{\ell}\cdot \nabla)\mathring{R}_{\ell},\nonumber
\end{align}
thus by (\ref{e:elementary inequality}) and the parameter assumption (\ref{e:parameter assumption 2}), we deduce that
\begin{align}
&\|(\partial_t+v_{\ell}\cdot \nabla) a_{kl}\|_N\nonumber\\
\leq& \frac{C}{\sqrt{\delta_{q+1}}}\Big(\Big\|\gamma_k\Big({\rm Id}-\frac{\mathring{R}_{\ell}}{\rho_l}\Big)\Big\|_N \|(\partial_t+v_{\ell}\cdot \nabla)\mathring{R}_{\ell}\|_0+\Big\|\gamma_k\Big({\rm Id}-\frac{\mathring{R}_{\ell}}{\rho_l}\Big)\Big\|_0 \|(\partial_t+v_{\ell}\cdot \nabla)\mathring{R}_{\ell}\|_N\Big)\nonumber\\
\leq& \frac{C(N)}{\sqrt{\delta_{q+1}}}\big(\delta_{q+1}\lambda_q\ell^{-N} +\delta_{q+1}\ell^{-N-1}\big)\leq C(N)\delta^{\frac12}_{q+1}\ell^{-N-1},
\quad \forall N\geq 0.\nonumber
\end{align}
Hence by (\ref{e: estimate on known velocity}), (\ref{e:elementary inequality}) and (\ref{e:parameter assumption 1}), we deduce that
\begin{align}
\|D_tL_{kl}\|_N\leq& \|(\partial_t+v_{\ell}\cdot \nabla) a_{kl}\|_N+\|a_{kl} \nabla v_{\ell} \nabla\Phi_l\|_N\nonumber\\[3pt]
\leq& C(N)\delta^{\frac12}_{q+1}\ell^{-N-1}+ C(N)\big(\|a_{kl}\|_N\|\nabla v_{\ell}\|_0\|\nabla \Phi_l\|_0+\|a_{kl}\|_0\|\nabla v_{\ell}\|_N\|\nabla \Phi_l\|_0\nonumber\\[3pt]
& \quad \quad \quad \quad +\|a_{kl}\|_0\|\nabla v_{\ell}\|_0\|\nabla \Phi_l\|_N\big)\nonumber\\
\leq &C(N)\delta^{\frac12}_{q+1}\ell^{-N-1}+ C(N)\delta_q^{\frac12}\delta_{q+1}^{\frac12}\lambda^2_q \ell^{1-N}
+C(N)\delta_q^{\frac12}\delta_{q+1}^{\frac12}\lambda_q \ell^{-N}\nonumber\\[3pt]
\leq& C(N)\delta_{q+1}^{\frac12} \ell^{-N-1}. \nonumber
\end{align}

{\bf Estimate on $\phi_{kl}$:} By (\ref{e:elementary inequality}) and (\ref{e:estimate on transport phase}), we obtain
\begin{align}
\|\phi_{kl}\|_1\leq& C\lambda_{q+1}\|\nabla \Phi_l-{\rm Id}\|_0\leq C\delta_q^{\frac12}\lambda_q \lambda_{q+1}\mu^{-1}, \nonumber\\[3pt]
\|\phi_{kl}\|_N\leq& C(N)\big(\lambda_{q+1}\|\nabla\Phi_l\|_{N-1}+(\lambda_{q+1}\|\nabla \Phi_l-{\rm Id}\|_0)^N\big)\nonumber\\
\leq & C(N)\lambda_{q+1}\delta_q^{\frac12}\lambda_q\mu^{-1} \ell^{-N+1}+C(N)\big(\delta_q^{\frac12}\lambda_q\lambda_{q+1}\mu^{-1}\big)^N, \quad \forall N\geq 2.\nonumber
\end{align}
Furthermore, by (\ref{e:parameter assumption 1}), we arrive at
\begin{align}
\|\phi_{kl}\|_N
\leq C(N)\lambda_{q+1}^{N(1-\beta)}, \quad \forall N\geq 1.\nonumber
\end{align}

{\bf Estimate on $w_o, w_c$:} Recalling (\ref{e:another formulation for main perturbation}), by (\ref{e:estimate on various amplitude}), we deduce that
\begin{align}
\|w_o\|_1\leq& C \sum_{l\in {\rm Z}}\sum_{k\in \Lambda_{(l)}} \chi_l(t)\|a_{kl} \phi_{kl}e^{i\lambda_{q+1} k^\perp\cdot x}\|_1\nonumber\\
\leq &C \sum_{l\in {\rm Z}}\sum_{k\in \Lambda_{(l)}} \chi_l(t)\big(\|a_{kl}\|_1+ \|a_{kl}\|_0\|\phi_{kl}\|_1+\|a_{kl}\|_0\lambda_{q+1}\big)\nonumber\\
\leq & \frac{M}{2}\delta_{q+1}^{\frac12}\lambda_{q+1}+C\delta_{q+1}^{\frac12}\lambda_{q+1}^{1-\beta}, \nonumber\\
\|w_o\|_N\leq& C \sum_{l\in {\rm Z}}\sum_{k\in \Lambda_{(l)}} \chi_l(t)\|a_{kl} \phi_{kl}e^{i\lambda_{q+1} k^\perp\cdot x}\|_N\nonumber\\
\leq & C(N) \sum_{l\in {\rm Z}}\sum_{k\in \Lambda_{(l)}} \chi_l(t)\big(\|a_{kl}\|_N+ \|a_{kl}\|_0\|\phi_{kl}\|_N+\|a_{kl}\|_0\lambda^N_{q+1}\big)\nonumber\\
\leq& C(N)\delta_{q+1}^{\frac12}\lambda^N_{q+1}, \quad N\geq 2. \nonumber
\end{align}
Recalling (\ref{d:difinition of corrctor}) and (\ref{e:parameter assumption 1}), we know that
\begin{align}
\|w_c\|_N\leq& C(N) \sum_{l\in {\rm Z}}\sum_{k\in \Lambda_{(l)}} \chi_l(t)\lambda_{q+1}^{-1}\Big(\|a_{kl}\|_{N+1}+\lambda_{q+1}\|a_{kl}(\nabla \Phi_l-{\rm Id})\|_{N}\nonumber\\
&\quad \quad \quad +\big[\|a_{kl}\|_{1}+\lambda_{q+1}\|a_{kl}(\nabla \Phi_l-{\rm Id})\|_0\big]\|e^{i\lambda_{q+1}\Phi_l\cdot k^\bot}\|_N\Big)\nonumber\\
\leq & C(N) \lambda_{q+1}^{-1}\Big(\delta_{q+1}^{\frac12}\lambda_q\ell^{-N}
+\lambda_{q+1}\delta_{q+1}^{\frac12}\ell^{-N}\lambda_{q}\delta_{q}^{\frac12}\mu^{-1}
+\lambda_{q+1}\delta_{q+1}^{\frac12}\lambda_{q}\delta_{q}^{\frac12}\mu^{-1}\lambda_{q+1}^{N}\Big)\nonumber\\
\leq & C(N)\delta_{q+1}^{\frac12}\frac{\delta_q^{\frac12}\lambda_q}{\mu}\lambda^N_{q+1}, \quad N\geq 0. \nonumber
\end{align}
\end{proof}

\subsection{Estimate on temperature perturbation}
Next, we estimate the difference of temperature $\theta_1-\theta$. A direct computation gives that
\begin{align}\label{e:equation for temperature difference}
\left\{
\begin{array}{ll}
\partial_t (\theta_1-\theta) + v_1\cdot\nabla (\theta_1-\theta)-\triangle (\theta_1-\theta)=-(v_1-v)\cdot\nabla\theta,\\[3pt]
(\theta_1-\theta)(0,x)=0.
\end{array}
\right.
\end{align}
Taking the $L^2$ inner product with $\theta_1-\theta$ and integrating by parts, we arrive at
\begin{align}
& \|(\theta_1-\theta)(t,\cdot)\|_{L^2}^2+2\int_0^t\|\nabla(\theta_1-\theta)(s,\cdot)\|_{L^2}^2ds\nonumber\\
=&2\int_0^t\int_{{\rm T}^2}(v_1-v)\cdot\nabla(\theta_1-\theta)\theta(s,x) dxds\nonumber\\
\leq& \int_0^t\|\nabla(\theta_1-\theta)(s,\cdot)\|_{L^2}^2ds+4\|\theta\|^2_0 \|v_1-v\|^2_0.\nonumber
\end{align}
Thus, by (\ref{e:estimate on temperature 1}), there holds
\begin{align}\label{e:energy difference estimate 0}
\|(\theta_1-\theta)(t,\cdot)\|_{L^2}^2+\int_0^t\|\nabla(\theta_1-\theta)(s,\cdot)\|_{L^2}^2ds\leq 4\|\theta^0\|^2_0 \|v_1-v\|^2_0.
\end{align}
Moreover,
\begin{align}
-(v_1-v)\cdot\nabla\theta=-\sum_l \sum_{k\in \Lambda_{(l)}}\chi(t)\phi_{kl}L_{kl}\cdot \nabla\theta e^{i\lambda_{q+1} k^\bot\cdot x}:=-\sum_l \sum_{k\in \Lambda_{(l)}}b_{kl} e^{i\lambda_{q+1} k^\bot\cdot x}. \nonumber
\end{align}
By (\ref{e:estimate on general source term}), to establish decay estimate for $\|\theta_1-\theta\|_{L^\infty L^2}$, we need the estimates $\|\nabla^N\theta\|_{L^\infty L^2_x}$. To this end, we establish the following lemma.
\begin{lemma}
Let $\theta$ be a smooth solution of the following transport-diffusion equation
\begin{align}\label{e:transport-diffusion equation for Cauchy problem}
\left\{
\begin{array}{ll}
\partial_t\theta+ v\cdot\nabla \theta- \triangle \theta=0, \\[3pt]
{\rm div}v=0,\\[3pt]
\theta(0,x)=\theta^0(x).
\end{array}
\right.
\end{align}
Then, there holds
\begin{align}\label{e:estimate on various temperature}
\|\nabla\theta\|_{L^\infty L^2}\leq& C\|v\|_0^2 \|\theta^0\|_2, \nonumber\\
 \|\nabla^N\theta\|_{L^\infty L^2}\leq&  C(N,\|v\|_0, \|\theta^0\|_2) \sum_{k=1}^{N-1}\|\nabla^k v\|_0^{\frac{N}{k+1}}, \quad \forall N\geq 2.
\end{align}
\begin{proof}
{\bf Step 1:}
Acting $\partial_j$ on both sides of (\ref{e:transport-diffusion equation for Cauchy problem}), we obtain
\begin{align}\label{e:transport equation for j-derivative}
\partial_t\partial_j\theta+ v\cdot\nabla(\partial_j\theta)- \triangle \partial_j\theta=-\partial_jv\cdot\nabla \theta.
\end{align}
Taking $L^2$ inner products with $\partial_j\theta$ and integrating by parts, we deduce that
\begin{align}
\frac12 \frac{d}{dt}\|\partial_j\theta(t,\cdot)\|_{L^2}^2+ \|\nabla\partial_j\theta(t,\cdot)\|_{L^2}^2
	\leq &\Big|\int_{{\rm T}^2}[v\cdot \nabla\partial_j\theta\partial_j\theta+v\cdot\nabla\theta \partial_{jj}\theta]dx \Big| \nonumber\\
\leq & \|v\|_0 \|\nabla\theta(t,\cdot)\|_{L^2}\|\partial_{jj}\theta(t,\cdot)\|_{L^2}.\nonumber
\end{align}
Summing about $j=1,2$, we arrive at
\begin{align}
\frac12 \frac{d}{dt}\|\nabla\theta(t,\cdot)\|_{L^2}^2+ \|\nabla^2\theta(t,\cdot)\|_{L^2}^2
	\leq  \|v\|_0 \|\nabla\theta(t,\cdot)\|_2\|\triangle\theta(t,\cdot)\|_2.\nonumber
\end{align}
Moreover, by the interpolation inequality $\|\nabla\theta(t,\cdot)\|_{L^2}\leq \|\theta(t,\cdot)\|^{\frac12}_{L^2} \|\nabla^2\theta(t,\cdot)\|^{\frac12}_{L^2}$, we arrive at
\begin{align}
\frac12 \frac{d}{dt}\|\nabla\theta(t,\cdot)\|_{L^2}^2+ \|\nabla^2\theta(t,\cdot)\|_{L^2}^2
	\leq  \|v\|_0 \|\theta(t,\cdot)\|^{\frac12}_{L^2}\|\nabla^2\theta(t,\cdot)\|^{\frac32}_{L^2}.\nonumber
\end{align}
Thus, by the H\"{o}lder inequality and noticing the fact $\|\theta(t,\cdot)\|_{L^2}\leq\|\theta^0\|_2$, we obtain
\begin{align}\label{e:estimate on gradient of temperature}
\|\nabla\theta\|_{L^\infty L^2}\leq C\|v\|_0^2 \|\theta^0\|_2,
\end{align}
which gives the first estimate in (\ref{e:estimate on various temperature}).

{\bf Step 2:}
For $N\geq 2$, acting $\nabla^N $ on both sides of (\ref{e:transport equation for j-derivative}), taking $L^2$ inner products with $\nabla^N \theta$ and integrating by parts, we obtain
\begin{align}
\frac12 \frac{d}{dt}\|\nabla^N\theta(t,\cdot)\|_{L^2}^2+ \|\nabla^{N+1}\theta(t,\cdot)\|_{L^2}^2
	\leq & \sum_{k=1}^{N-1}C_N^k\|\nabla^k v\|_0 \|\nabla^{N+1-k}\theta(t,\cdot)\|_{L^2}\|\nabla^{N}\theta(t,\cdot)\|_{L^2} \nonumber\\
&+ \Big|\int_{{\rm T}^2}\nabla^N v\cdot \nabla \theta \cdot \nabla^N\theta\Big|.\nonumber
\end{align}
Integrating by parts and using interpolation inequality
\begin{align}
\|\nabla^m\theta\|_{L^2}\leq \|\nabla \theta\|^{\frac{N+1-m}{N}}_{L^2}\|\nabla^{N+1}\theta\|^{\frac{m-1}{N}}_{L^2}, \quad 2\leq m\leq N, \nonumber
\end{align}
we deduce that
\begin{align}
\Big|\int_{{\rm T}^2}\nabla^N v\cdot \nabla \theta \cdot \nabla^N\theta\Big|\leq& C \|\nabla^{N-1}v\|_0 \|\nabla \theta\|_{L^2}\|\nabla^{N+1}\theta\|_{L^2},\nonumber\\
\|\nabla^{N+1-k}\theta(t,\cdot)\|_{L^2}\|\nabla^{N}\theta(t,\cdot)\|_{L^2}\leq & \|\nabla \theta(t,\cdot)\|_{L^2}^{\frac{k+1}{N}}\|\nabla^{N+1}\theta(t,\cdot)\|^{\frac{2N-k-1}{N}}_{L^2}.\nonumber
\end{align}
Thus, by Young inequality, we obtain
\begin{align}
\frac12 \frac{d}{dt}\|\nabla^N\theta(t,\cdot)\|_{L^2}^2+ \|\nabla^{N+1}\theta(t,\cdot)\|_{L^2}^2
	\leq & C(N)\sum_{k=1}^{N-1}\|\nabla^k v\|_0^{\frac{2N}{k+1}}\|\nabla\theta(t,\cdot)\|^2_2+ \frac12\|\nabla^{N+1}\theta(t,\cdot)\|_{L^2}^2.\nonumber
\end{align}
Combining the estimate (\ref{e:estimate on gradient of temperature}), we obtain
\begin{align}
\|\nabla^N\theta\|_{L^\infty L^2}\leq C(N)\|v\|_0^2 \|\theta^0\|_2\sum_{k=1}^{N-1}\|\nabla^k v\|_0^{\frac{N}{k+1}}, \nonumber
\end{align}
which is the second estimate in (\ref{e:estimate on various temperature}).
\end{proof}
\end{lemma}

\begin{Corollary}
Under the assumption (\ref{e: estimate on known velocity}), for any $N\geq 2$, there holds
\begin{align}\label{e:estimate on n-derivative on velocity}
\|\nabla^N\theta\|_{L^\infty L^2}\leq C(N)\delta_q^{\frac12}\lambda_q^{N-1}.
\end{align}
\end{Corollary}
\begin{proof}
By (\ref{e: estimate on known velocity}) and (\ref{e:estimate on various temperature}), we deduce that
\begin{align}
\|\nabla^N\theta\|_{L^\infty L^2}\leq C(N,\|v\|_0^2, \|\theta^0\|_2)\sum_{k=1}^{N-1}\delta_q^{\frac{N}{2(k+1)}}\lambda_q^{\frac{Nk}{k+1}}\leq C(N)\delta_q^{\frac12}\lambda_q^{N-1}.\nonumber
\end{align}
\end{proof}

\begin{Proposition}
There holds
\begin{align}\label{e:estimate on gradient of temperature difference}
\|\nabla^N(\theta_1-\theta)\|_{L^\infty L^2_x}\leq& C\delta_{q+1}^{\frac12}\lambda^{N-1}_{q+1},\quad N\geq 0.
\end{align}
Furthermore, there holds
\begin{align}\label{e:fractional derivative estimate on temperature difference}
\|\theta_1-\theta\|_{L^\infty \dot{H}^\varepsilon_x}\leq& C \delta_{q+1}^{\frac12}\lambda_{q+1}^{\varepsilon-1}, \quad \varepsilon\geq 0.
\end{align}
\end{Proposition}
\begin{proof}
{\bf Estimate on $\|\theta_1-\theta\|_{L^\infty L^2_x}$:}
Recalling $b_{kl}=\chi_l(t)\phi_{kl}L_{kl}\cdot \nabla\theta$, thus by (\ref{e: estimate on known velocity}), (\ref{e:estimate on various amplitude}), (\ref{e:estimate on various temperature}), (\ref{e:estimate on n-derivative on velocity}) and the parameter assumption (\ref{e:parameter assumption 1}), we deduce that for any $l\in Z$ and $k\in \Lambda_0\cup \Lambda_1$, there holds
\begin{align}\label{e:estimate on spatial bkl}
\|b_{kl}\|_{L^\infty L^2_x}\leq& C \delta^{\frac12}_{q+1},\nonumber\\
\|\nabla^N b_{kl}\|_{L^\infty L^2_x}\leq& C(N)\Big(\|L_{kl}\|_N \|\phi_{kl}\|_0 \|\nabla \theta\|_{L^\infty L^2_x}\nonumber\\
& \quad \quad \quad \quad \quad+\|L_{kl}\|_0 \|\phi_{kl}\|_N \|\nabla \theta\|_{L^\infty L^2_x}+\|L_{kl}\|_0 \|\phi_{kl}\|_0 \|\nabla^N\nabla \theta\|_{L^\infty L^2_x}\Big)\nonumber\\
\leq & C(N) \delta_{q+1}^{\frac12}\big(\ell^{-N}+\delta_q^{\frac12}\lambda_q\lambda_{q+1}\mu^{-1}\ell^{-(N-1)}+(\delta^{\frac12}_q\lambda_q \lambda_{q+1}\mu^{-1})^N+\delta_q^{\frac12}\lambda_q^N\big)\nonumber\\
\leq & C(N) \delta_{q+1}^{\frac12}\big(\ell^{-N}+\delta_q^{\frac12}\lambda_q\lambda_{q+1}\mu^{-1}\ell^{-(N-1)}+(\delta^{\frac12}_q\lambda_q \lambda_{q+1}\mu^{-1})^N\big).
\end{align}
Thus, by (\ref{e:estimate on general source term}) and the parameter assumption (\ref{e:parameter assumption 1}), we obtain
\begin{align}
\|\theta_1-\theta\|_{L^\infty L^2_x}\leq C(N)\delta_{q+1}^{\frac12}\lambda_{q+1}^{-1}+ C(N)\frac{\delta_{q+1}^{\frac12}\lambda_{q+1}^{(N+1)(1-\beta)}}{\lambda_{q+1}^{N+1}}\leq C(N)\delta_{q+1}^{\frac12}(\lambda_{q+1}^{-1}+\lambda_{q+1}^{-N\beta}). \nonumber
\end{align}
Taking $N$ large enough such that $N\beta\geq 1$, we arrive at
\begin{align}
\|\theta_1-\theta\|_{L^\infty L^2_x}\leq C\delta_{q+1}^{\frac12}\lambda_{q+1}^{-1}. \nonumber
\end{align}

{\bf Estimate on $\|\nabla^N(\theta_1-\theta)\|_{L^\infty L^2_x}$($N\geq 1$):}
Recalling (\ref{e: estimate on known velocity}), we deduce that
\begin{align}
\sum_{k=0}^{N-1}\|\nabla^k v\|_0^{\frac{N+1}{k+1}}\leq C\sum_{k=0}^{N-1} \delta_q^{\frac{N+1}{2(k+1)}}\lambda_q^{\frac{k(N+1)}{k+1}}
\leq C\delta_q^{\frac12}\lambda_q^N.\nonumber
\end{align}
By (\ref{e:estimate on spatial bkl}) and the parameter assumption (\ref{e:parameter assumption 1}), we deduce that
\begin{align}
\|\nabla^{N-1} b_{kl}\|_{L^\infty L^2}\leq C(N)\delta_{q+1}^{\frac12}\lambda_{q+1}^{(N-1)(1-\beta)}.\nonumber
\end{align}
Thus, by (\ref{e:estimate on general source term}) and the parameter assumption (\ref{e:parameter assumption 1}), we obtain
\begin{align}
\|\nabla^N(\theta_1-\theta)\|_{L^\infty L^2_x}\leq C(N)\delta_{q+1}^{\frac12}\lambda_{q+1}^{N-1},\nonumber
\end{align}
this completes the proof.
\end{proof}

\setcounter{equation}{0}

\section{Estimate on the Reynold stress}

In this section, we prove the following estimate for the Reynold stress.

\begin{Proposition}
The Reynold stress $\mathring{R}_1$ defined in (\ref{d:difinition of Reynold stress}) satisfies the estimate
\begin{align}\label{e:estimate on the new Reynold stress}
\|\mathring{R}_1\|_0+\frac{\|\mathring{R}_1\|_1}{\lambda_{q+1}}
\leq C(\varepsilon) \delta_{q+1}^{\frac12}\Big((\mu+\ell^{-1})\lambda_{q+1}^{-1+\varepsilon}+\delta_q^{\frac12}\lambda_q\ell
+\delta_q^{\frac12}\lambda_q \lambda_{q+1}^\varepsilon \mu^{-1}\Big).
\end{align}
\end{Proposition}

\begin{proof}We deal with it term by term. \\
{\bf Estimate on $\mathcal{R}^0$:} A direct computation gives
\begin{align}
(\partial_t+v_{\ell}\cdot\nabla)w=\sum_{l\in {\rm Z}}\sum_{k\in \Lambda_{(l)}}\Big(\chi_l'(t) L_{kl}+\chi(t)(\partial_t L_{kl}+v_{\ell}\cdot\nabla L_{kl})\Big)\phi_{kl}e^{i\lambda_{q+1} k^\bot\cdot x}.\nonumber
\end{align}
Set
\begin{align}
\Psi_{kl}:=\chi_l'(t) L_{kl}+\chi(t)(\partial_t L_{kl}+v_{\ell}\cdot\nabla L_{kl}), \nonumber
\end{align}
thus, by $|\chi_l'(t)|\leq C\mu$ and (\ref{e:estimate on various amplitude}), we deduce that for any $l\in Z$ and $k\in \Lambda_0\cup \Lambda_1$, $|\mu t-l|< 1$, there hold
\begin{align}
\|\Psi_{kl}\|_N\leq& C\mu \|L_{kl}\|_N+\|\partial_t L_{kl}+v_{\ell}\cdot\nabla L_{kl}\|_N\nonumber\\
 \leq & C(N) \delta_{q+1}^{\frac12}\ell^{-N}\big(\mu+\ell^{-1}\big), \quad \forall N\geq 0. \nonumber
\end{align}
By (\ref{e:oscillatory estimate}), we deduce that
\begin{align}
\|\mathcal{R}^0\|_0 \leq C(N,\varepsilon) & \sum_{l\in {\rm Z}}\sum_{k\in \Lambda_{(l)}}\Big( \frac{\|\Psi_{kl}\|_0}{\lambda^{1-\varepsilon}_{q+1}}
+\frac{\|\Psi_{kl}\|_N}{\lambda_{q+1}^{N-\varepsilon}}+\frac{\|\Psi_{kl}\|_{N+\varepsilon}}{\lambda_{q+1}^N}\Big)\nonumber\\
\leq & C(N,\varepsilon)\Big(\delta_{q+1}^{\frac12}\lambda_{q+1}^{-1+\varepsilon}(\mu+\ell^{-1})+\delta_{q+1}^{\frac12}\ell^{-N}(\mu+\ell^{-1}) \lambda_{q+1}^{-N+\varepsilon}\Big).\nonumber
\end{align}
Combining the parameter assumption (\ref{e:parameter assumption 1}) and taking $N$ large enough such that $N\beta> 1$, we arrive at
\begin{align}\label{e:estimate on first error}
\|\mathcal{R}^0\|_0 \leq C(\varepsilon)\delta_{q+1}^{\frac12}(\mu+\ell^{-1})\lambda_{q+1}^{-1+\varepsilon}.
\end{align}
The same argument gives
\begin{align}\label{e:estimate on first error graident}
\|\mathcal{R}^0\|_1 \leq C(\varepsilon)\delta_{q+1}^{\frac12}(\mu+\ell^{-1})\lambda_{q+1}^{\varepsilon}.
\end{align}
{\bf Estimate on $\mathcal{R}^1$:} A direct computation gives
\begin{align}
w\cdot\nabla v_{\ell}=\sum_{l\in {\rm Z}}\sum_{k\in \Lambda_{(l)}}\chi(t)L_{kl}\cdot\nabla v_{\ell}\phi_{kl}e^{i\lambda_{q+1} k^\bot\cdot x}.\nonumber
\end{align}
Set
\begin{align}
\Omega_{kl}=\chi(t)L_{kl}\cdot\nabla v_{\ell}\phi_{kl}, \nonumber
\end{align}
thus, (\ref{e:estimate on various amplitude}) and (\ref{e:elementary inequality}) implies that for any $l\in Z$ and $k\in \Lambda_0\cup \Lambda_1$, $|\mu t-l|< 1$,
\begin{align}
\|\Omega_{kl}\|_0\leq& C \delta_{q+1}^{\frac12}\delta_q^{\frac12}\lambda_q, \nonumber\\
 \|\Omega_{kl}\|_N\leq& C(N)\big(\|L_{kl}\|_N\|\nabla v_{\ell}\|_0\|\phi_{kl}\|_0+\|L_{kl}\|_0\|\nabla v_{\ell}\|_N\|\phi_{kl}\|_0\nonumber\\[3pt]
 & \quad \quad \quad \quad \quad \quad \quad \quad \quad + \|L_{kl}\|_0\|\nabla v_{\ell}\|_0\|\phi_{kl}\|_N\big)\nonumber\\
 \leq & C \delta_{q+1}^{\frac12}\delta_q^{\frac12}\lambda_q\Big(\ell^{-N}+\lambda_{q+1}\delta_q^{\frac12}\lambda_q\mu^{-1}\ell^{-N+1}
 +(\lambda_{q+1}\delta_q^{\frac12}\lambda_q\mu^{-1})^N\Big), \quad \forall N\geq 1. \nonumber
\end{align}
Combining the parameter assumption (\ref{e:parameter assumption 1}), we obtain
\begin{align}
\|\Omega_{kl}\|_N\leq C(N) \delta_{q+1}^{\frac12}\delta_q^{\frac12}\lambda_q \lambda_{q+1}^{N(1-\beta)}. \nonumber
\end{align}
By taking $N$ sufficiently large, a similar argument as for $\mathcal{R}^0$ gives
\begin{align}\label{e:estimate on second error}
\|\mathcal{R}^1\|_0 \leq & C(\varepsilon)\sum_{l\in {\rm Z}}\sum_{k\in \Lambda_{(l)}}\Big( \frac{\|\Omega_{kl}\|_0}{\lambda^{1-\varepsilon}_{q+1}}
+\frac{\|\Omega_{kl}\|_N}{\lambda_{q+1}^{N-\varepsilon}}
+\frac{\|\Omega_{kl}\|_{N+\varepsilon}}{\lambda_{q+1}^N}\Big)\nonumber\\
\leq & C(\varepsilon)\delta_{q+1}^{\frac12}\delta_q^{\frac12}\lambda_q\lambda_{q+1}^{-1+\varepsilon}, \nonumber\\[3pt]
\|\mathcal{R}^1\|_1 \leq &C(\varepsilon)\delta_{q+1}^{\frac12}\delta_q^{\frac12}\lambda_q\lambda_{q+1}^{\varepsilon}.
\end{align}
{\bf Estimate on $\mathcal{R}^2$:} By Lemma \ref{l:anti-divergence operator} and (\ref{e:estimate on gradient of temperature difference}), we deduce that
\begin{align}\label{e:estimate on third part}
\|\mathcal{R}^2\|_0\leq& C\|\theta_1-\theta\|_{L^\infty \dot{H}^\varepsilon_x}\leq C\delta_{q+1}^{\frac12}\lambda^{\varepsilon-1}_{q+1},\nonumber\\
\|\mathcal{R}^2\|_1\leq& C\|\nabla(\theta_1-\theta)\|_{L^\infty \dot{H}^\varepsilon_x}\leq C\delta_{q+1}^{\frac12}\lambda_{q+1}^\varepsilon.
\end{align}
{\bf Estimate on $\mathcal{R}^3$:} Recalling (\ref{e:formular of oscillatory term}), we know
\begin{align}
\mathcal{R}^3=\mathcal{R}\Big({\rm div}\Big(w_o\otimes w_o-\sum_l\chi^2_l R_{\ell,l}+P{\rm Id}\Big)\Big)=\mathcal{R}\Big(T^1_{osc}+T^2_{osc}\Big).\nonumber
\end{align}
We first deal with $T^1_{osc}$. By (\ref{e:estimate on various amplitude}) and the parameter assumption (\ref{e:parameter assumption 2}), we deduce that for any $l\in Z$ and $k\in \Lambda_0\cup \Lambda_1$, $|\mu t-l|< 1$, there holds
\begin{align}\label{e:estimate on amplitude of nonlinear interaction}
\|a_{kl}a_{k'l}\phi_{kl}\phi_{k'l}\|_1\leq& \|a_{kl}a_{k'l}\|_1\|\phi_{kl}\phi_{k'l}\|_0+ \|a_{kl}a_{k'l}\|_0 \|\phi_{kl}\phi_{k'l}\|_1 \nonumber\\[3pt]
\leq & C \delta_{q+1}\big(\lambda_q+\delta_q^{\frac12}\lambda_q \lambda_{q+1}\mu^{-1}\big)\leq C \delta_{q+1}\delta_q^{\frac12}\lambda_q \lambda_{q+1}\mu^{-1}, \nonumber\\[3pt]
\|a_{kl}a_{k'l}\phi_{kl}\phi_{k'l}\|_N\leq& C(N)\big(\|a_{kl}a_{k'l}\|_N\|\phi_{kl}\phi_{k'l}\|_0+ \|a_{kl}a_{k'l}\|_0 \|\phi_{kl}\phi_{k'l}\|_N\big)\nonumber\\[3pt]
 \leq & C(N)\Big(\delta_{q+1}\lambda_q \ell^{1-N}+\delta_{q+1}\lambda_{q+1}^{N(1-\beta)}\Big), \quad \forall N\geq2.
\end{align}
Recalling (\ref{e:definition of T1}), by (\ref{e:oscillatory estimate}), choosing $N$ sufficiently large,  we deduce that
\begin{align}
\|\mathcal{R}(T^1_{osc})\|_0\leq & C(\varepsilon)\sum_{l \in {\rm Z}}\sum_{k,k'\in \Lambda_{(l)},k+k'\neq 0}\chi_l^2\Big(\frac{\| \nabla(a_{kl}a_{k'l}\phi_{kl}\phi_{k'l})\|_0}{\lambda_{q+1}^{1-\varepsilon}}\nonumber\\
&\quad \quad \quad +\frac{\| \nabla(a_{kl}a_{k'l}\phi_{kl}\phi_{k'l})\|_N}{\lambda_{q+1}^{N-\varepsilon}}+\frac{\| \nabla(a_{kl}a_{k'l}\phi_{kl}\phi_{k'l})\|_{N+\varepsilon}}{\lambda_{q+1}^N}\Big)\nonumber\\
\leq & C(\varepsilon)\delta_{q+1}\delta_q^{\frac12}\lambda_q \lambda_{q+1}^\varepsilon \mu^{-1}. \nonumber
\end{align}
Similarly, we can obtain
\begin{align}
\|\mathcal{R}(T^1_{osc})\|_1\leq  C(\varepsilon)\delta_{q+1}\delta_q^{\frac12}\lambda_q \lambda_{q+1}^{1+\varepsilon} \mu^{-1}. \nonumber
\end{align}
Next, we deal with the second term $T^2_{osc}$.
From (\ref{d:difinition on f}) and (\ref{e:estimate on various amplitude}), we deduce that
\begin{align}
\|f_{klk'l'}\|_N\leq& \chi_l(t)\chi_{l'}(t) \|a_{kl}a_{k'l'}\phi_{kl}\phi_{k'l'}\|_N.\nonumber
\end{align}
Thus, the same argument as above gives that
\begin{align}
\|\mathcal{R}(T^2_{osc})\|_0\leq& C(\varepsilon)\delta_{q+1}\delta_q^{\frac12}\lambda_q \lambda_{q+1}^\varepsilon \mu^{-1},\nonumber\\
\|\mathcal{R}(T^2_{osc})\|_1\leq& C(\varepsilon)\delta_{q+1}\delta_q^{\frac12}\lambda_q \lambda_{q+1}^{1+\varepsilon} \mu^{-1}.\nonumber
\end{align}
Summing the two parts, we obtain
\begin{align}\label{e:estimate on fourth error}
\|\mathcal{R}^3\|_0+\frac{\|\mathcal{R}^3\|_1}{\lambda_{q+1}} \leq C(\varepsilon)\delta_{q+1}\delta_q^{\frac12}\lambda_q \lambda_{q+1}^\varepsilon \mu^{-1}.
\end{align}
{\bf Estimate on $\mathcal{R}^4$:} By (\ref{e:estimate on velocity perturbation}), we deduce
\begin{align}\label{e:estimate on fifth error}
\|\mathcal{R}^4\|_0 \leq& C \|w_o\|_0\|w_c\|_0\leq C \delta_{q+1}\delta_q^{\frac12}\lambda_q  \mu^{-1}, \nonumber\\
\|\mathcal{R}^4\|_1 \leq& C \big(\|w_o\|_1\|w_c\|_0+\|w_o\|_0\|w_c\|_1\big)\leq C \delta_{q+1}\delta_q^{\frac12}\lambda_q \lambda_{q+1} \mu^{-1}.
\end{align}
{\bf Estimate on $\mathcal{R}^5$:} By (\ref{e:estimate on velocity perturbation}), (\ref{e:estimate on convolution 2}), (\ref{e:estimate on convolution 1}) and the parameter assumption (\ref{e:parameter assumption 2}), we deduce
\begin{align}\label{e:estimate on sixth error}
\|\mathcal{R}^5\|_0 \leq& C \|w\|_0\|v-v_{\ell}\|_0\leq C \delta_{q+1}^{\frac12}\delta_q^{\frac12}\lambda_q  \ell, \nonumber\\
\|\mathcal{R}^5\|_1 \leq& C \big(\|w\|_1\|v-v_{\ell}\|_0+\|w\|_0\|v-v_{\ell}\|_1\big)\leq C \delta_{q+1}^{\frac12}\delta_q^{\frac12}\lambda_q \lambda_{q+1} \ell.
\end{align}
{\bf Estimate on $\mathcal{R}^6$:}  By (\ref{e:estimate on convolution 1}), we deduce
\begin{align}\label{e:estimate on seventh error}
\|\mathcal{R}^6\|_0 \leq& C \|\nabla \mathring{R}\|_0 \ell \leq C \delta_{q+1}\lambda_q  \ell,\nonumber\\
\|\mathcal{R}^6\|_1 \leq& C \|\nabla \mathring{R}\|_1  \leq C \delta_{q+1}\lambda_q  .
\end{align}
Finally, summing estimates (\ref{e:estimate on first error})-(\ref{e:estimate on seventh error}), we obtain (\ref{e:estimate on the new Reynold stress}).
\end{proof}

\setcounter{equation}{0}

\section{Proof of Proposition \ref{p: iterative 1}}

\begin{proof}
{\bf Step 1: Choice of the parameters $\mu, \ell$.}
Recalling that the sequence $\{\delta_q\}_{q\in {\rm N}}$ and $\{\lambda_q\}_{q\in {\rm N}}$ are chosen to satisfy
\begin{align}
\delta_q=a^{-b^q}, \quad \lambda_q\in\big[a^{cb^{q+1}}, 2a^{cb^{q+1}}\big], \nonumber
\end{align}
where $b=\frac{6+\gamma}{4}, c=\frac{4(5+\gamma)}{6+\gamma}$ and $a>1$ which will be determined later.
Take
\begin{align}\label{e:parameter choosing}
\mu=\delta_q^{\frac14}\lambda_q^{\frac12}\lambda_{q+1}^{\frac12}, \quad \ell=\delta_q^{-\frac14}\lambda_{q}^{-\frac12}\lambda_{q+1}^{-\frac12}.
\end{align}
We check that $\mu, \ell, \delta_q, \lambda_q$ satisfies (\ref{e:parameter assumption 1}). Firstly, by (\ref{e:parameter choosing}), it's easy to obtain
\begin{align}
\frac{\delta_q^{\frac12}\lambda_q \ell}{\delta_{q+1}^{\frac12}}=\frac{\mu}{\lambda_{q+1}\delta_{q+1}^{\frac12}}=
\frac{\lambda_{q}^{\frac12}\delta_{q}^{\frac14}}{\lambda_{q+1}^{\frac12}\delta_{q+1}^{\frac12}}, \quad \frac{\delta_q^{\frac12}\lambda_q}{\mu}=\frac{1}{\ell \lambda_{q+1}}=\frac{\lambda_{q}^{\frac12}\delta_{q}^{\frac14}}{\lambda_{q+1}^{\frac12}}.\nonumber
\end{align}
Taking $a$ large enough, a direct computation gives that
\begin{align}
\lambda_{q}^{\frac12}\delta_{q}^{\frac14}\lambda_{q+1}^{-\frac12}\delta_{q+1}^{-\frac12}\leq & 2a^{\big(-\frac14+\frac{cb}{2}+\frac{b}{2}-\frac{cb^2}{2}\big)b^q}\leq 2a^{\frac{-6-6\gamma-\gamma^2}{8}b^q}\leq 2a^{-\frac34b^q}\leq 1, \nonumber\\
\lambda_{q}^{\frac12}\delta_{q}^{\frac14}\lambda_{q+1}^{-\frac12}\leq & 2a^{\big(-\frac{1}{4cb^2}+\frac{1}{2b}-\frac{1}{2}\big)cb^{q+2}}\leq 2a^{-\frac{2+\gamma}{12+\gamma}cb^{q+2}}\leq \frac12\lambda_{q+1}^{-\beta}, \nonumber
\end{align}
where we fix $\beta=\frac18$. This justifies the parameter assumption (\ref{e:parameter assumption 1}).

{\bf Step 2: Proof of (\ref{e:Reynold stress estimate})-(\ref{e:energy estimate on temperature}).} Firstly, (\ref{e:estimate on the new Reynold stress}) and (\ref{e:parameter choosing}) imply
\begin{align}
\|\mathring{R}_{q+1}\|_0+\frac{\|\mathring{R}_{q+1}\|_1}{\lambda_{q+1}}\leq C\delta_{q+1}^{\frac12}\delta_q^{\frac14}\lambda_q^{\frac12}\lambda_{q+1}^{-\frac12+\varepsilon}+
C\delta_{q+1}^{\frac12}\delta_q^{\frac14}\lambda_q^{\frac12}\lambda_{q+1}^{-\frac12}\leq C\delta_{q+1}^{\frac12}\delta_q^{\frac14}\lambda_q^{\frac12}\lambda_{q+1}^{-\frac12+\varepsilon}.\nonumber
\end{align}
A direct computation gives
\begin{align}
\delta_{q+1}^{\frac12}\delta_q^{\frac14}\lambda_q^{\frac12}\lambda_{q+1}^{-\frac12+\varepsilon}\leq & 2a^{-b^{q+2}\big(\frac{1}{2b} +\frac{1}{4b^2}-\frac{c}{2b}+(\frac12-\varepsilon)c\big)}\nonumber\\
\leq & 2a^{-b^{q+2}\big(1+\frac{4\gamma+\gamma^2}{(6+\gamma)^2}-\frac{4(5+\gamma)}{6+\gamma}\varepsilon\big)} \leq\eta \delta_{q+2},\nonumber
\end{align}
where we take $a$ sufficiently large and $\varepsilon=\varepsilon(\gamma)$ sufficiently small. This gives  (\ref{e:Reynold stress estimate}).

By (\ref{e:estimate on velocity perturbation}), we deduce
\begin{align}
\|v_{q+1}-v_q\|_0\leq& \|w_o\|_0+\|w_c\|_0\leq \delta_{q+1}^{\frac12}\Big(\frac{M}{2}+\lambda_{q+1}^{-\beta}\Big),\nonumber\\
\|v_{q+1}-v_q\|_N\leq& \|w_o\|_N+\|w_c\|_N\leq C(N)\delta_{q+1}^{\frac12}\lambda^N_{q+1}, \quad \forall N\geq 1.\nonumber
\end{align}
Noticing $\lambda_q\geq \lambda_1\geq a^{cb^2}$, thus we obtain (\ref{e:velocity difference estimate}) by taking $a$ sufficiently large.


Recalling (\ref{e:perturbation on pressure}), by (\ref{e:estimate on various amplitude}) and the support property of $\chi$, we obtain
\begin{align}
\|P\|_0\leq \frac{M^2}{2} \delta_{q+1}, \quad \|P\|_1\leq \frac{M^2}{2} \delta_{q+1}\lambda_{q+1}. \nonumber
\end{align}
 Thus, from (\ref{d:definition of new pressure}), (\ref{e:estimate on velocity perturbation}) and (\ref{e:parameter choosing}), we deduce
\begin{align}
\|p_{q+1}-p_q\|_0\leq& \frac{M^2}{2} \delta_{q+1}+2\big(\|w_o\|_0\|w_c\|_0+\|w\|_0\|v-v_{\ell}\|_0\big)\leq M^2\delta_{q+1},\nonumber\\
\|p_{q+1}-p_q\|_1\leq& \frac{M^2}{2} \delta_{q+1}\lambda_{q+1}+2\big(\|w_o\|_1\|w_c\|_0+\|w_o\|_0\|w_c\|_1\nonumber\\[3pt]
& \quad \quad \quad +\|w\|_1\|v-v_{\ell}\|_0+\|w\|_0\|v-v_{\ell}\|_1\big)\leq M^2\delta_{q+1}\lambda_{q+1},\nonumber
\end{align}
this give (\ref{e:pressure difference estimate}).

Moreover, (\ref{e:energy conservation of temperature}) implies (\ref{e:energy identity}). Using the above estimate on $\|v_{q+1}-v_q\|_0$, we easily get (\ref{e:energy estimate on temperature}) from (\ref{e:energy difference estimate 0}).

{\bf Step 3: Proof of (\ref{e:time derivative estimate}).} Recall that $v_{q+1}=v_q+w$ and
\begin{align}
\partial_t w=\sum_{l\in {\rm Z}}\sum_{k\in \Lambda_{(l)}}(\chi'_l(t) L_{kl}+\chi_l(t)\partial_t L_{kl}+i\lambda_{q+1}\chi_l(t)L_{kl}k^\perp\cdot \partial_t\Phi_l)e^{i\lambda_{q+1} k^\perp\cdot \Phi_l}.\nonumber
\end{align}
From (\ref{e:estimate on transport phase}), (\ref{e:estimate on various amplitude}) and the parameter assumption (\ref{e:parameter assumption 1}), we deduce that
\begin{align}
\|\partial_t w\|_0\leq C \delta_{q+1}^{\frac12}(\mu+\ell^{-1}+\lambda_{q+1})\leq C \delta_{q+1}^{\frac12}\lambda_{q+1},\nonumber
\end{align}
hence
\begin{align}
\|\partial_t (v_{q+1}-v_q)\|_0\leq  C \delta_{q+1}^{\frac12}\lambda_{q+1}.\nonumber
\end{align}
Similarly, there holds
\begin{align}
\|\partial_t (p_{q+1}-p_q)\|_0\leq C \delta_{q+1}\lambda_{q+1}, \nonumber
\end{align}
this completes the proof of (\ref{e:time derivative estimate}).

{\bf Step 4: Proof of (\ref{e:energy difference estimate}).} Finally, to complete the proof of this proposition, we only need to justify (\ref{e:energy difference estimate}).
A direct computation gives
\begin{align}
&e(t)\big(1-\delta_{q+2}\big)-\int_{{\rm T}^2}|v_{q+1}(t,x)|^2dx\nonumber\\
=& \underbrace{e(t)\big(1-\delta_{q+2}\big)-\int_{{\rm T}^2}|v_q(t,x)|^2dx-\int_{{\rm T}^2}|w_o|^2dx}_{Err_1}\nonumber\\
&-\underbrace{2\int_{{\rm T}^2}v_q\cdot wdx- \int_{{\rm T}^2}(2 w_o\cdot w_c+|w_c|^2)dx}_{Err_2}.\nonumber
\end{align}
We first deal with $Err_2$. From (\ref{e:estimate on velocity perturbation}), it's easy to obtain
\begin{align}
\Big|\int_{{\rm T}^2}\big(2 w_o\cdot w_c+|w_c|^2\big)dx\Big|\leq C\frac{\delta_{q+1}\delta_q^{\frac12}\lambda_q}{\mu}.\nonumber
\end{align}
Moreover,
\begin{align}
v_q\cdot w=\sum_l \sum_{k\in \Lambda_{(l)}} \chi(t)L_{kl}\cdot v_q\phi_{kl}e^{i\lambda_{q+1} k^\bot\cdot x}.\nonumber
\end{align}
Thus, by (\ref{e:estimate on various amplitude}) and (\ref{e:oscillatory integration estimate}),  we deduce that
\begin{align}
\Big|\int_{{\rm T}^2}v_q\cdot wdx\Big|=&\Big|\sum_l \sum_{k\in \Lambda_{(l)}} \chi(t)\int_{{\rm T}^2}L_{kl}\cdot v_q\phi_{kl}e^{i\lambda_{q+1} k^\bot\cdot x}dx\Big|\nonumber\\
\leq & \sum_l \sum_{k\in \Lambda_{(l)}} \chi(t)\frac{\|\nabla(L_{kl}\cdot v_q\phi_{kl})\|_0}{\lambda_{q+1}}\nonumber\\
\leq & C \frac{\delta^{\frac12}_{q+1}\delta_q^{\frac12}\lambda_q}{\mu}+C\frac{\delta_{q+1}^{\frac12}\ell^{-1}}{\lambda_{q+1}}.\nonumber
\end{align}
Combining the two parts, we obtain
\begin{align}\label{e:estimate on second energy error}
|Err_2|\leq C \frac{\delta^{\frac12}_{q+1}\delta_q^{\frac12}\lambda_q}{\mu}+C\frac{\delta_{q+1}^{\frac12}\ell^{-1}}{\lambda_{q+1}}.
\end{align}
For $Err_1$, we first compute $\int_{{\rm T}^2}|w_o|^2dx$. From (\ref{c:computation of nonlinear interaction}), (\ref{c:computation of nonlinear interaction in same index}) and (\ref{c:computation of nonlinear interaction in different index}), we deduce that
\begin{align}
|w_0|^2=&2\sum_l\chi^2_l(t)\rho_l-\sum_{l \in {\rm Z}}\sum_{k,k'\in \Lambda_{(l)},k+k'\neq 0}\chi_l^2 a_{kl}a_{k'l}\phi_{kl}\phi_{k'l}k\cdot k'e^{i\lambda_{q+1}(k+k')^\bot\cdot x}\nonumber\\
&-\sum_{l\neq l', k\in \Lambda_{(l)}, k'\in \Lambda_{(l')}}f_{klk'l'} k\cdot k'e^{i\lambda_{q+1}(k+k')^\bot\cdot x}:=2\sum_l\chi^2_l(t)\rho_l+I_1+I_2.\nonumber
\end{align}
Recalling (\ref{d:definition of l amplitude}), we obtain
\begin{align}
2(2\pi)^2\sum_l\chi^2_l(t)\rho_l=&e(t)\big(1-\delta_{q+2}\big)-\int_{{\rm T}^2}|v_q(t,\cdot)|^2dx\nonumber\\
&+\underbrace{\sum_l \chi^2_l(t)\Big[\Big(e\Big(\frac{l}{\mu}\Big)-e(t)\Big)\big(1-\delta_{q+2}\big)}_{I_{31}}\nonumber\\
&+\underbrace{\int_{{\rm T}^2}\Big(|v_q(t,x)|^2-\Big|v_q\Big(\frac{l}{\mu},x\Big)\Big|^2\Big)dx\Big]}_{I_{32}}.\nonumber
\end{align}
It's easy to obtain that
\begin{align}
|I_{31}|\leq C(e)\mu^{-1}.\nonumber
\end{align}
As in \cite{BCDLI}, using the equation (\ref{e:boussinesq-reynold equation}), we can deduce
\begin{align}
&\int_{{\rm T}^2}\Big(|v_q(t,x)|^2-|v_q(l\mu^{-1},x)|^2\Big)dx=\int^{t}_{\frac{l}{\mu}}\int_{{\rm T}^2}\partial_t |v_q(s,x)|^2dsdx\nonumber\\
=&-\int^{t}_{\frac{l}{\mu}}\int_{{\rm T}^2}{\rm div}\big(v_q(|v_q|^2+2p_q)(s,x)\big)dsdx+2\int^{t}_{\frac{l}{\mu}}\int_{{\rm T}^2}v_q\cdot{\rm div}\mathring{R}_q(s,x) dsdx\nonumber\\
&+\int^{t}_{\frac{l}{\mu}}\int_{{\rm T}^2}\theta_q e_2\cdot v_q (s,x) dsdx\nonumber\\
=&-2\int^{t}_{\frac{l}{\mu}}\int_{{\rm T}^2}\nabla v_q:\mathring{R}_q(s,x) dsdx+\int^{t}_{\frac{l}{\mu}}\int_{{\rm T}^2}\theta_q e_2\cdot v_q (s,x) dsdx;\nonumber
\end{align}
thus, by (\ref{e: estimate on known velocity}) and (\ref{e: estimate on known Reynold stress}), for any $l\in Z$ and $t$ in the range $|\mu t-l|< 1$, there hold
\begin{align}
|I_{32}|\leq C \frac{\delta_{q+1}\delta_q^{\frac12}\lambda_q}{\mu}+ \frac{C}{\mu}.\nonumber
\end{align}
Moreover, by (\ref{e:estimate on amplitude of nonlinear interaction}) and (\ref{e:oscillatory integration estimate}), we deduce
\begin{align}
\Big|\int_{{\rm T}^2}(I_1+I_2)dx\Big|\leq C \frac{\delta_{q+1}\delta_q^{\frac12}\lambda_q}{\mu}.\nonumber
\end{align}
Thus, collecting these estimate together, we obtain
\begin{align}\label{e:estimate on first energy error}
|Err_1|\leq C(e)\mu^{-1}+C\delta_{q+1}\delta_q^{\frac12}\lambda_q\mu^{-1}.
\end{align}
Combining (\ref{e:estimate on second energy error}) and (\ref{e:estimate on first energy error}), we deduce that
\begin{align}
&\Big|e(t)\big(1-\delta_{q+2}\big)-\int_{{\rm T}^2}|v_{q+1}(t,x)|^2dx\Big|\leq C(e)\mu^{-1}+C\delta_{q+1}^{\frac12}\delta_q^{\frac12}\lambda_q\mu^{-1}+C\frac{\delta_{q+1}^{\frac12}\ell^{-1}}{\lambda_{q+1}}.\nonumber
\end{align}
Recalling (\ref{e:parameter choosing}), we obtain
\begin{align}
\Big|e(t)\big(1-\delta_{q+2}\big)-\int_{{\rm T}^2}|v_{q+1}(t,x)|^2dx\Big|\leq & C\delta_q^{-\frac14}\lambda_q^{-\frac12}\lambda_{q+1}^{-\frac12}+C\delta_q^{\frac14}\delta_{q+1}^{\frac12}\lambda_q^{\frac12}\lambda_{q+1}^{-\frac12}\nonumber\\
\leq & C\delta_q^{\frac14}\delta_{q+1}^{\frac12}\lambda_q^{\frac12}\lambda_{q+1}^{-\frac12}\leq Ca^{\big(-\frac14+\frac{b}{2}+\frac{cb}{2}-\frac{cb^2}{2}\big)}\delta_{q+1}\nonumber\\[3pt]
\leq&  Ca^{-\frac12 b^{q}}\delta_{q+1},\nonumber
\end{align}
which implies (\ref{e:energy difference estimate}) by taking $a$ sufficiently large.

{\bf Step 5: Choice of the $(v_0, p_0, \theta_0, \mathring{R}_0)$.} Set
\begin{align}
v_0(t,x):=&\binom{\sqrt{\frac{(1-\delta_1)e(t)}{2\pi^2}}\sin(\lambda_0 x_2)}{0},  \nonumber\\
\mathring{R}_0:=&
\left(\begin{array}{cc}
0 & -\frac{\sqrt{1-\delta_1}e'(t)}{\sqrt{8\pi^2 e(t)}} \frac{\cos(\lambda_0 x_2)}{\lambda_0} \\
-\frac{\sqrt{1-\delta_1}e'(t)}{\sqrt{8\pi^2 e(t)}} \frac{\cos(\lambda_0 x_2)}{\lambda_0} & 0
\end{array}
\right),\nonumber\\[3pt]
 p_0:=&e^{-t} \int_0^{x_2}\theta^0(s)ds, \quad \theta_0:=e^{-t}\theta^0(x_2),\nonumber
\end{align}
where $\theta^0\in X $ as Theorem \ref{t:main 1}.

It's direct to justify that $(v_0, p_0, \theta_0, \mathring{R}_0)$ satisfies the Boussinesq-Reynold system (\ref{e:boussinesq-reynold equation}) and
\begin{align}
\Big|e(t)(1-\delta_1)-\int_{{\rm T}^2}|v_0(t,x)|^2dx\Big|=&0\leq \frac{\delta_1}{4}e(t), \nonumber\\
\frac12 \|\theta_0(t,\cdot)\|_2^2+\int_0^t \|\nabla \theta_0(t,\cdot)\|_{L^2}^2dx=&\frac12 \|\theta^0(\cdot)\|^2_{L^2}, \quad \forall t\geq 0. \nonumber
\end{align}
Moreover, by taking $\lambda_0$ sufficiently large, we obtain
\begin{align}
\|\mathring{R}_0\|_0\leq \eta \delta_1. \nonumber
\end{align}

Finally, staring with $(v_0, p_0, \theta_0, \mathring{R}_0)$, we can construct sequence $(v_q, p_q, \theta_q, \mathring{R}_q)$ inductively which solve Boussinesq-Reynold system (\ref{e:boussinesq-reynold equation}) and satisfy (\ref{e:Reynold stress estimate})-(\ref{e:time derivative estimate}). Thus, we complete the proof of Proposition \ref{p: iterative 1}.

\end{proof}

\appendix


\setcounter{equation}{0}
\section{Estimate for transport equation}

In this section we give some well known estimates for the smooth solution of transport equation:
\begin{align}\label{e:transport equation on torus}
\left\{
\begin{array}{ll}
\partial_tf+v\cdot\nabla f=g,\\
f|_{t=0}=f_0
\end{array}
\right.
\end{align}
where $v=v(t,x)$ is a given smooth vector field. The proof can be found in \cite{BCDLI}.

\begin{Proposition}
Assume $|t|\|v\|_1\leq 1$. Then any solution of (\ref{e:transport equation on torus}) satisfies the following estimates:
\ben
\|f(t)\|_0&\leq& \|f_0\|_0+\int_0^t\|g(\cdot, \tau)\|_0d\tau, \label{e:transport estimate for zero derivative}\\
\|f(t)\|_\alpha&\leq& 2\Big( \|f_0\|_\alpha+\int_0^t\|g(\cdot, \tau)\|_\alpha d\tau\Big) \label{e:transport estimate for holder derivative}
\een
for all $0<\alpha\leq 1$. Generally, for any $N\geq 1$ and $0\leq\alpha\leq 1$, there hold
\ben
[f(\cdot, t)]_{N+\alpha}\lesssim  [f_0]_{N+\alpha}+|t|[v]_{N+\alpha}[f_0]_1+\int_0^t\big([g(\cdot, \tau)]_{N+\alpha}+(t-\tau)[v]_{N+\alpha}[g(\cdot, \tau)]_1 \big)d\tau. \label{e:transport estimate for high derivative}
\een
Let $\Phi(t,\cdot)$ be the inverse of the flux $X$ of $v$ staring at time $t_0$ as the identity. Under the above assumption $|t-t_0|\|v\|_1\leq 1$ we have
\ben
\|\nabla \Phi-{\rm Id}\|_0\lesssim |t-t_0|[v]_1,\quad [\nabla \Phi]_N\lesssim |t-t_0|[v]_N,\quad \forall N\geq 1.\label{e:inverse flux estimate}
\een
\end{Proposition}

\section{Stationary phase estimate}

In the section, we recall the following simple fact, and the proof can also be found in \cite{BCDLI}.

\begin{Lemma}
Let $k\in Z^2, k\neq 0$ and $\lambda\geq 1$ be given. For any $a\in C^\infty({\rm T}^2)$ and $m\in {\rm N}$, there hold
\begin{align}\label{e:oscillatory integration estimate}
\Big|\int_{{\rm T}^2}a(x)e^{i\lambda k\cdot x}dx\Big|\leq \frac{[a]_m}{\lambda^m}.
\end{align}
\end{Lemma}

{\bf Acknowledgments.}
The first author is supported in part by NSFC Grants 11601258.6.
The second author is supported by the fundamental research funds of Shandong university under Grant 11140078614006. The third author is partially supported by the Chinese NSF under Grant 11471320 and 11631008.


\begin{thebibliography}{WWW}



\bibitem{TBU}
T. Buckmaster, {\it Onsager's conjecture almost everywhere in time},
Comm. Math. Phys. 333(2015), 1175-1198.

\bibitem{BCV}
T.Buckmaster, M. Colombo and V.Vicol, {\it Wild solutions of the Navier-Stokes equations whose singular sets in time have Hausdorff dimension strictly less than 1},  arXiv:1809.00600

\bibitem{BCDLI}
T. Buckmaster, C.  De Lellis, P.  Isett,  and Sz\'{e}kelyhidi. Jr.
L, {\it Anomalous dissipation for 1/5-H\"{o}lder Euler
flows},  Ann. of. Math. 182(2015), 127-172

\bibitem{BCDL1}
T. Buckmaster, C.  De Lellis, and Sz\'{e}kelyhidi. Jr.
L, {\it Transporting microstructure and dissipative Euler flows},
arXiv:1302.2825, 2013

\bibitem{BCDL2}
T. Buckmaster, C.  De Lellis, and Sz\'{e}kelyhidi. Jr.L, {\it Dissipative Euler
flows with Onsager-critical spatial
regularity},  Comm. Pure Appl. Math, 69(2016), 1613-1670

\bibitem{BCDLV}
T. Buckmaster, C.  De Lellis,  Sz\'{e}kelyhidi. Jr.L, and V. Vicol, {\it Onsager conjecture for admissible weak solution}, Comm. Pure Appl. Math, 72(2019), 229-274

\bibitem{BSV}
T. Buckmaster, Shkoller, and V. Vicol, {\it Nonuniqueness of weak solutions to SQG equation}, to appear in Comm. Pure Appl. Math.

\bibitem{BV}
T. Buckmaster, and V. Vicol, {\it Nonuniqueness of weak solutions to Navier-Stokes equation}, Ann. of. math, 189(2019), 101-144

\bibitem{Chae}
D. Chae, {\it Global regularity for the 2-D Boussinesq equation with partial viscous terms}, Adv. in. Math, 203(2006), 497-513

\bibitem{CPFR}
A. Cheskidov, P. Constantin, S. Friedlander, and R. Shvydkoy, {\it Energy conservation and Onsager's conjecture for the Euler
equations},  Nonlinearity  21(6)(2008), 1233-1252

\bibitem{CHO}
A. Choffrut, {\it H-principles for the incompressible Euler
equations},  Arch. Ration. Mech. Anal. 210(2013), 133-163.

\bibitem{PC}
P. Constantin, {\it On the Euler equation of incompressible flow},
Bull. Amer. Math. Soc. 44(4)(2007), 603-621.



\bibitem{CM}
P. Constantin, and A. Majda, {\it The Beltrami spectrum for incompressible fluid flows},
 Comm. Math. Phys. 115(1988), 435-456

\bibitem{CCDE}
S. Conti, C. De Lellis, and Sz\'{e}kelyhidi. Jr.
L,  {\it H-principle and
rigidity for $C^{1,\alpha}$ isometric embeddings} , In Nonlinear Partial Differential
 Equations vol.7 of Abel Symposia Springer (2012), 83-116.


\bibitem{CET}
P. Constantin, E. W, and Titi. E. S, {\it Onsager's conjecture on
the energy conservation for solutions of Euler's equation}, Comm. Math. Phys,
 165(1)(1994), 207-209.



\bibitem{DA}
S. Daneri, {\it Cauchy problem for dissipative H\"{o}lder solutions to the
incompressible Euler equations},  Comm. Math. Phy. (2014), 1-42.

\bibitem{DAL}
S. Daneri, and Sz\'{e}kelyhidi. Jr.
L, {\it Non-uniqueness and h-principle for H\"{o}lder continuous weak solution of Euler equation},  Arch. Ration. Mech. Anal, 224(2017), 471-514


\bibitem{CDL}
C. De Lellis, and Sz\'{e}kelyhidi. Jr.
L, {\it The Euler equation as a differential inclusion}, Ann. of. Math. 170(3)(2009), 1417-1436.

\bibitem{CDL0}
C. De Lellis, and Sz\'{e}kelyhidi. Jr.
L, {\it On admissibility criteria
for weak solutions of the Euler equations},  Arch. Ration. Mech. Anal. 195(1)(2010), 225-260.


\bibitem{CDL1}
C. De Lellis, and Sz\'{e}kelyhidi. Jr.
L,   {\it The h-principle and the
equations of fluid dynamics}, Bull. Amer. Math. Soc. 49(3)(2012), 347-375.


\bibitem{CDL2}
C. De  Lellis, and Sz\'{e}kelyhidi. Jr.L, {\it Dissipative continuous
Euler flows}, Invent. Math. 193(2)(2013), 377-407


\bibitem{CDL3}
C. De Lellis, and Sz\'{e}kelyhidi. Jr.
L, {\it Dissipative Euler flows and
Onsager's conjecture}, Jour. Eur. Math. Soc.(JEMS)16(2014), 1467-1505.

\bibitem{CDDL}
C. De Lellis, D. Inauen, and Sz\'{e}kelyhidi. Jr.
L, {\it A Nash-Kuiper theorem for $C^{1,\frac{1}{5}-\kappa}$ immersions of surfaces in 3 dimensions},
 Rev. Mat. Iberoam. 34(2018), 1119-1152

\bibitem{DUR}
J. Duchon, and R. Raoul, {\it Inertial energy dissipation for weak solutions of incompressible Euler and Navier-Stokes equations},
Nonlinearity. 13(2000), 249-255


\bibitem{Hou}
T. Hou and C. Li, {\it Global well-posedness of the viscous Boussinesq equations}, DCDS, Series A, 12(2005), 1-12.


\bibitem{ISOH1}
P. Isett, and Oh, S.-J, {\it A heat flow approach to Onsager's conjecture for the Euler equations on manifolds},
Trans. Amer. Math. Soc, 368(2016), 6519-6537.

\bibitem{ISOH2}
P. Isett, Oh, S.-J, {\it On nonperiodic Euler flows with H\"{o}lder regularity}, Arch. Ration. Mech. Anal, 221(2016), 725–804.


\bibitem{IS1}
P. Isett,  {\it H\"{o}lder continuous Euler
flows in three dimensions with
compact support in time}, arXiv:1211.4065, 2012.

\bibitem{IS2}
P. Isett,  {\it A proof of Onsager's conjecture}, Ann. of. Math, 188(2018), 871-963.

\bibitem{IS3}
P. Isett, {\it On the Endpoint Regularity in Onsager's Conjecture},  arXiv:1706.01549

\bibitem{ISV2}
P. Isett, and V. Vicol, {\it H\"{o}lder continuous solutions of active scalar equations},
 Ann. of. PDE, DOI 10.1007/s40818-015-0002-0


\bibitem{LT}
T. Luo, and Titi, {\it Non-uniqueness of Weak Solutions to Hyperviscous Navier-Stokes Equations - On Sharpness of J.-L. Lions Exponent},  arXiv:1808.07595

\bibitem{LTZ}
T. Luo, T. Tao and L.Zhang, {\it Finite energy weak solutions of 2d Boussinesq Equations with diffusive temperature}, arXiv:1901.09179

\bibitem{LX}
T. Luo, and Z. Xin , {\it H\"{o}lder continuous solutions to the 3d Prandtl system}, arXiv:1804.04285

\bibitem{L}
Xiaoyutao Luo, {\it Stationary solution and nonuniquenes of weak solution for the Navier-Stokes euation on high dimensions}, Arch. Ration. Mech. Anal, 233(2019), 701-747

\bibitem{Ma}
A. Majda,  {\it  Introduction to PDEs and Waves for the Atmosphere and Ocean},
Courant
Lecture Notes in Mathematics, Vol. 9. AMS/CIMS, 2003

\bibitem{MS0}
S.Modena and G.Sattig, {\it Convex solutions to the transport equation with full dimensional concentration}, arXiv:1902.08521

\bibitem{MS1}
S.Modena and L\'{a}szl\'{o} Sz\'{e}kelyhidi Jr, {\it Non-uniqueness for the transport equation with Sobolev vector fields}, to appear in Ann. of. PDE

\bibitem{MS2}
S.Modena and L\'{a}szl\'{o} Sz\'{e}kelyhidi Jr, {\it Non-Renormalized solution to the continuity equation}, arXiv:1806.09145

\bibitem{NASH}
J. Nash, {\it $C^{1}$ isometric embeddings},  Ann. of. Math. 60(1954), 383-396.

\bibitem{Nov}
M.D.Novack, {\it Non-uniqueness of weak solutions to the 3D QG equations}, arXiv:1812.08734

\bibitem{ONS}
L. Onsager, {\it Statistical hydrodynamics}, Nuovo Cimento(9)(1949), 279-287.

\bibitem{Pe}
J. Pedlosky, {\it Geophysical fluid dynamics}, Springer, New-York, 1987

\bibitem{VS}
V. Scheffer,  {\it An inviscid flow with compact support in space-time},
J. Geom. Anal. (1993),343-401.


\bibitem{ASH1}
A. Shnirelman, {\it Weak solution with decreasing energy of incompressible Euler equations}, Comm. Math. Phys. 210(2000), 541-603

\bibitem{ASH2}
A. Shnirelman, {\it On the nonuniqueness of weak solution of Euler equation}, Comm. Pure Appl. Math. 50(12)(1997), 1261-1286







\bibitem{RSH}
R. Shvydkoy, {\it Convex integration for a class of active scalar equations}, J. Amer. Math. Soc. 24(4)(2011), 1159-1174


\bibitem{SH}
R. Shvydkoy, {\it Lectures on the Onsager conjecture},  Dis. Con. Dyn. Sys. 3(3)(2010), 473-496.

\bibitem{Sz}
Jr. L. Sz\'{e}kelyhidi, {\it From Isometric Embeddings to Turbulence}, Lecture note, 2012.

\bibitem{TZ}
T. Tao, and L. Zhang, {\it On the continuous periodic weak solution of Boussinesq equations}, SIAM, J.Math.Anal, 50(2018), 1120-1162

\bibitem{TZ1}
T. Tao, and L. Zhang, {\it H\"{o}lder continuous solution of Boussinesq equations with compact support}, J. Funct. Anal, 272(2017), 4334-4402.

\bibitem{TZ2}

T. Tao, and L. Zhang, {\it H\"{o}lder continuous periodic solution of Boussinesq equations with partial viscosity},  Calc. Var. Partial Differential Equations (2018)
https://doi.org/10.1007/s00526-018-1337-7



\end{thebibliography}
\end{document}